\documentclass[leqno]{amsart}
\usepackage{cases}
\usepackage{cases}
\usepackage{amsmath}
\usepackage{graphicx}
\usepackage{mathrsfs}
\usepackage{bbm}
\usepackage{amsfonts}
\usepackage{amssymb}
\usepackage{amsmath,amsthm}
\usepackage{hyperref}
\usepackage{xcolor}
\usepackage{color}
\usepackage{ifpdf}
\usepackage{url}
\usepackage{enumerate}
\usepackage{multirow}
\usepackage{diagbox}
\usepackage{rotating}
\usepackage{tikz}

\definecolor{grey}{rgb}{0.5,0.5,0.4}
\def\opn#1#2{\def#1{\operatorname{#2}}} 
\opn\chara{char} \opn\length{\ell}
\opn\projdim{proj\,dim} \opn\injdim{inj\,dim} \opn\rank{rank}
\opn\depth{depth} \opn\grade{grade} \opn\height{height}
\opn\embdim{emb\,dim} \opn\codim{codim}

\opn\Tr{Tr} \opn\bigrank{big\,rank}
\opn\superheight{superheight}\opn\lcm{lcm}
\opn\trdeg{tr\,deg}%
\opn\reg{reg} \opn\lreg{lreg}
\opn\Ker{Ker} \opn\Coker{Coker} \opn\Im{Im} \opn\Hom{Hom}
\opn\Tor{Tor} \opn\Ext{Ext} \opn\End{End} \opn\Aut{Aut} \opn\id{id}

\opn\nat{nat}
\opn\pff{pf}
\opn\Pf{Pf} \opn\GL{GL} \opn\SL{SL} \opn\mod{mod} \opn\ord{ord}

\def\Implies{\ifmmode\Longrightarrow \else
     \unskip${}\Longrightarrow{}$\ignorespaces\fi}
\def\implies{\ifmmode\Rightarrow \else
     \unskip${}\Rightarrow{}$\ignorespaces\fi}
\def\iff{\ifmmode\Longleftrightarrow \else
     \unskip${}\Longleftrightarrow{}$\ignorespaces\fi}

\let\:=\colon

\DeclareMathOperator*\lowlim{lim\ inf}
\DeclareMathOperator*\uplim{lim\ sup}

\newtheorem{theorem}{Theorem}[section]
\newtheorem{lemma}[theorem]{Lemma}

\newtheorem{Theorem}{Theorem}[section]

\newtheorem{Corollary}[Theorem]{Corollary}

\newtheorem{Remark}[Theorem]{Remark}

\newtheorem{Example}[Theorem]{Example}

\newtheorem{Definition}[Theorem]{Definition}

\newtheorem{Assumption}[Theorem]{Assumption}

\theoremstyle{definition}

\opn\ini{in} \opn\inm{inm} \opn\Sym{Sym} \opn\diag{diag}
\opn\Ii{(i)} \opn\Iii{(ii)}

\title{ On Lagrange multipliers of constrained optimization in Hilbert spaces}

\author{Zhiyu Tan $^\star$}

\thanks{ Zhiyu Tan, School of Mathematical Sciences and Fujian Provincial Key Laboratory on Mathematical Modeling and High Performance Scientific Computing, Xiamen University, Fujian, 361005, China.
Email: {\tt zhiyutan@xmu.edu.cn} }

\begin{document}
\maketitle

{\bf Abstract:}\hspace*{10pt} {In this paper we introduce the essential Lagrange multiplier and establish the solid mathematical foundation of constrained optimization in Hilbert spaces with sharp results on the mathematical foundation of quadratic-programming based methods such as the SQP method, the  necessary and sufficient conditions for the existence and uniqueness of Lagrange multipliers, the essential difference of the theory of Lagrange multipliers in finite and infinite-dimensional spaces and an essential characterization of the convergence of the classical augmented Lagrangian method. They are achieved by a newly developed decomposition framework for Lagrange multipliers of the Karush-Kuhn-Tucker system of constrained optimization problems in Hilbert spaces, which is totally different from the existing theories based on separation theorems.}

{{\bf Keywords:}\hspace*{10pt}  Constrained optimization,\ KKT system,\ Essential Lagrange multiplier,\ Quadratic-programming based methods,\ Augmented Lagrangian method.}
\vspace{8pt}


\section{Introduction}
Constrained optimization is ubiquitous in the real world. The development of optimization mainly benefits from its wide applications in engineering, computer science, economy and many other fields, and modern optimization theory is deeply affected by this. Many exciting results have been achieved in algorithm design and theoretical analysis in the past several decades (cf. \cite{bertsekas1996constrained, boyd2004convex, hinze2008optimization, ito2008lagrange, nocedal2006numerical}). However, there are still some fundamental mathematical issues of constrained optimization that remain unclarified, missing or even untouched. Just a few are listed below.
\begin{enumerate}[$\bullet$]
  \item What is the mathematical foundation of the QP-based methods, such as the SQP method?
  
  \item  What are the sufficient and necessary conditions for the existence and uniqueness of Lagrange multipliers? 
  
  \item  How to mathematically clarify the difference between the existence theory of Lagrange multipliers in finite and infinite-dimensional spaces?
  
  \item How to characterize the convergence of multipliers in the multiplier-based methods (such as the augmented Lagrangian method and ADMM) without assuming the existence of Lagrange multipliers?
      
\end{enumerate} 
Note that in order to answer these questions, a solid mathematical foundation of constrained optimization is necessary.

The KKT (Karush-Kuhn-Tucker) system and the related Lagrange multipliers are of great significance to the theories and algorithms of constrained optimization problems (cf. \cite{bertsekas1996constrained, bertsekas2009convex, hinze2008optimization, ito2008lagrange, karush1939minima, Kuhn2014, kuhn1951nonlinear, luenberger1997optimization, nocedal2006numerical, rockafellar1993lagrange, rockafellar1997convex}). The existing theory of Lagrange multipliers is mainly based on separation theorems (cf. \cite{hinze2008optimization, ito2008lagrange}) and it is far from  complete. For example, the existing theory does not answer the third question above. Note that this question arises naturally from optimal control problems with pointwise state constraints (cf. \cite{brenner2020p, hinze2008optimization, ito2008lagrange}). The requirement of interior point conditions in separation theorems makes it hard or even impossible to answer this question by following the existing approach, because interior point conditions are essential for separation theorems in infinite-dimensional spaces (cf. \cite[Chapter 1, Remark 4]{brezis2011functional}). 

In this paper we attempt to move away from the classical approach based on separation theorems to develop a new framework to investigate Lagrange multipliers of the KKT system of constrained optimization problems, which can also cover some cases of variational inequalities related to constrained optimization problems.  In our approach we first construct a surrogate model that should share the same KKT system with the optimization problem at a given local minimizer by making use of the linearization of the problem, and then discuss the KKT system of the surrogate model at the minimizer. The approach here is based on a key observation that the KKT system only involves the linearized information of the optimization problem at a minimizer and is somehow a reformulation of the first order necessary condition with respect to the linearizing cone at the minimizer.

For the surrogate model,  we will prove the following existence theorem (see Section \ref{sec:general}).
\begin{theorem}\label{thm:sm_e}
For a given minimizer $u^*$ of the constrained optimization problem (\ref{eq:GO}),  the surrogate model exists for all $f\in \mathcal{F}(u^*)$ at the minimizer if and only if Guignard's condition (\ref{eq:GCQ}) holds, where $\mathcal{F}(u^*)$ is the set of all Fr\'{e}chet differentiable objective functions which have a local constrained minimizer at $u^*$. 
\end{theorem}

In general the surrogate model is a special case of the following model problem
\begin{equation}\label{eq:sm_g}
\min \theta(u)\quad \mbox{s.t.}\ \ Su \in K,
\end{equation} 
where $\mathcal{U}$ and $\mathcal{X}$ are real Hilbert spaces, $\emptyset \neq K\subseteq \mathcal{X}$ is closed and convex, and $S$ is a bounded linear operator from $\mathcal{U}$ to $\mathcal{X}$. According to the Riesz representation theorem (cf. \cite[Theorem 3.4, Chapter I]{conway2007course}), we set $\mathcal{U}' = \mathcal{U}$ and $\mathcal{X}' = \mathcal{X}$. We assume that the feasible set $\mathrm{R}(S)\cap K\neq \emptyset$, where $\mathrm{R}(S)$ is the range of $S$ and $\theta(u)$ is continuously Fr\'{e}chet differentiable and strongly convex on $\mathcal{U}$, i.e., there exists $c_0>0$ such that 
\begin{equation}\label{eq:theta_sc}
\langle u - v, D_u\theta(u) - D_u\theta(v)\rangle_\mathcal{U} \geq c_0 \|u - v\|^2_\mathcal{U},
\end{equation}
where $D_u\theta(v)$ is the first order Fr\'{e}chet derivative of $\theta(\cdot)$ at $v$, $\langle\cdot, \cdot\rangle_\mathcal{U}$ is the inner product of $\mathcal{U}$ and $\|\cdot\|_{\mathcal{U}}$ is the induced norm. The inner product and the induced norm on $\mathcal{X}$ are denoted by $\langle\cdot, \cdot\rangle$ and $\|\cdot\|$ respectively.

It follows from the assumptions on the model problem (\ref{eq:sm_g}) and the classical convex optimization theory that there exists a unique global minimizer $u^*$ of the model problem (\ref{eq:sm_g}). We will investigate the KKT system of the model problem at $u^*$, and the results can be applied to the surrogate model. To avoid separation theorems, we use an optimization procedure regularization approach to derive the KKT system at $u^*$, which will be realized by the classical augmented Lagrangian method (ALM, for short)(cf. \cite{hestenes1969multiplier, powell1969method}) in this paper. By carrying out the convergence analysis of the classical ALM without using any information of Lagrange multipliers of the model problem (\ref{eq:sm_g}) at $u^*$, we will prove the following theorem (cf. Appendix A). 
\begin{theorem}\label{thm:w-akkt}
A feasible point $u^*$ is a global minimizer of the model problem (\ref{eq:sm_g}) if and only if there exists $\{\lambda^k\}_{k=1}^{+\infty}\subset \mathcal{X}$ such that the following weak form asymptotic KKT system (W-AKKT, for short) holds
\begin{equation}\label{eq:w-akkt}
\left\{
\begin{aligned}
& \langle D_u\theta(u^*), v\rangle_\mathcal{U}  + \lim\limits_{k\rightarrow +\infty}\langle \lambda^{k}, Sv\rangle = 0\quad \forall\ v\in \mathcal{U};\\
& - \uplim\limits_{k\rightarrow+\infty}\langle\lambda^{k} ,\zeta - Su^*\rangle\geq 0\quad \forall\ \zeta\in K.\\
\end{aligned}
\right.
\end{equation}
\end{theorem}

Motivated by the results in Theorem \ref{thm:w-akkt}, we introduce the essential Lagrange multiplier.
\begin{Definition}\label{def:elm}
An element $\lambda^*\in \overline{R(S)}$ is called an essential Lagrange multiplier of the model problem (\ref{eq:sm_g}) at $u^*$ if it satisfies
\begin{equation}\label{eq:elm}
\left\{
\begin{aligned}
& \langle D_u\theta(u^*), v\rangle_\mathcal{U}  + \langle \lambda^{*}, Sv\rangle = 0\quad \forall\ v\in \mathcal{U};\\
& - \langle\lambda^{*} ,\zeta - Su^*\rangle\geq 0\quad \forall\ \zeta\in \overline{K\cap R(S)},\\
\end{aligned}
\right.
\end{equation}
where $\overline{R(S)}$ and $\overline{K\cap R(S)}$ are the closure of $R(S)$ and $K\cap R(S)$, respectively.
\end{Definition}
\noindent We also recall the definition of the proper Lagrange multiplier and the classical KKT system as follows (cf. \cite[p. 160]{barbu2012convexity}).
\begin{Definition}\label{def:plm}
An element $\bar{\lambda}\in \mathcal{X}$ is called a proper Lagrange multiplier of the model problem (\ref{eq:sm_g}) at $u^*$ if it satisfies the classical KKT system
\begin{equation}\label{eq:plm}
\left\{
\begin{aligned}
& \langle D_u\theta(u^*), v\rangle_\mathcal{U}  + \langle \bar{\lambda}, Sv\rangle = 0\quad \forall\ v\in \mathcal{U};\\
& - \langle\bar{\lambda} ,\zeta - Su^*\rangle\geq 0\quad \forall\ \zeta\in K.\\
\end{aligned}
\right.
\end{equation}
\end{Definition}

Note that the essential Lagrange multiplier of the model problem (\ref{eq:sm_g}) is actually the proper Lagrange multiplier of the optimization problem
\begin{equation}\nonumber
\min \theta(u)\quad \mbox{s.t.}\ \ Su \in \overline{K\cap R(S)}\subset \overline{R(S)},
\end{equation} 
which implies that the essential Lagrange multiplier is only related to the feasible set of the model problem (\ref{eq:sm_g}). This observation inspires us to investigate the proper Lagrange multiplier with the help of the essential Lagrange multiplier. More details on the essential Lagrange multiplier will be given in Section \ref{sec:ELM}. These results indicate that the essential Lagrange multiplier is a fundamental concept in the theory of Lagrange multipliers of constrained optimization.

As an application of the essential Lagrange multiplier, we will consider the convergence of the multipliers generated by the classical ALM for the model problem (\ref{eq:sm_g}). Our results show some equivalence between the convergence of the multipliers and the existence of the essential Lagrange multiplier (see Theorem \ref{thm:ALM-ELM}). This indicates that the essential Lagrange multiplier is of fundamental importance in the application of Lagrange multipliers.

The rest of the paper is organized as follows. In Section \ref{sec:general} a general optimization problem and the related variational inequalities are given, and some results of the surrogate model are also presented there, especially the proof of Theorem \ref{thm:sm_e}. A thorough discussion of the essential Lagrange multiplier and the proper Lagrange multiplier of the model problem (\ref{eq:sm_g}) is included in Section \ref{sec:ELM}. The essential Lagrange multiplier of the model problem (\ref{eq:sm_g}) always exists in finite-dimensional spaces, while this is not the case in the infinite-dimensional spaces. The results in this section also theoretically confirm the necessity of using asymptotic or approximate KKT systems (cf \cite{andreani2019sequential, borgens2020new, jeyakumar2003new, thibault1997sequential}) to give the optimality conditions of constrained optimization problems in infinite-dimensional spaces. Section \ref{sec:w-akkt} is devoted to some further applications of the W-AKKT system. We give an elementary proof of the existence of the proper Lagrange  multiplier under a generalized Robinson's condition. Note that Robinson's condition  (see (\ref{eq:RCQ})) is widely used in the theories and applications of optimization problems in both finite and infinite-dimensional spaces (cf. \cite{ito2008lagrange, mangasarian1967fritz, maurer1979first, robinson1976stability, zowe1979regularity}). The paper ends up with some concluding remarks in Section \ref{sec:con}. Details for the proof of Theorem \ref{thm:w-akkt} are provided in Appendix \ref{append:A}. 

In this paper we use the standard notations from functional analysis and convex analysis, see for example in \cite{barbu2012convexity, bauschke2011convex, conway2007course, rockafellar1997convex, yosida2012functional}. 

\section{A general optimization problem and the surrogate model}\label{sec:general}
\setcounter{equation}{0}
Let us consider the following constrained optimization problem
\begin{equation}\label{eq:GO}
\min f(u)\quad \mbox{s.t.}\ \ G(u) \in \mathcal{K}, 
\end{equation} 
where $f:\ \mathcal{U}\rightarrow \mathbb{R}$, $G:\ \mathcal{U}\rightarrow \mathcal{X}$, $\mathcal{K}$ is a closed and convex set in $\mathcal{X}$, and $\mathcal{U}$, $\mathcal{X}$ are two real Hilbert spaces. Assume that $G(u)\cap \mathcal{K}\neq \emptyset$ and $u^*$ is a minimizer of the optimization problem (\ref{eq:GO}). We will investigate the KKT system of the optimization problem (\ref{eq:GO}) at $u^*$. We first assume that $f$ and $G$ are Fr\'{e}chet differentiable, and then we extend the results to some nonsmooth cases. As aforementioned, we will denote by $\mathcal{F}(u^*)$ the set of all Fr\'{e}chet differentiable objective functions which have a local constrained minimizer at $u^*$.

\subsection{Some preliminary results}
Let $M$ be the feasible set, i.e.,
\begin{equation}\label{eq:FS}
M = \{u\in\mathcal{U}:\ G(u)\in\mathcal{K} \}.
\end{equation} 
We  denote by $T(M, \bar{u})$, $T_w(M, \bar{u})$ and $L(\mathcal{K}, \bar{u})$ the sequential tangent cone,  the weak sequential tangent cone and the linearizing cone at $\bar{u}\in M$ respectively, which are defined by 
\begin{multline}\nonumber
T(M, \bar{u}) = \big\{v\in \mathcal{U}:\ \exists \{u_n\}_{n=1}^{+\infty}\subset M,\ \{t_n\}_{n=1}^{+\infty}\subset \mathbb{R}^+,\ u_n \rightarrow \bar{u}, \ t_n\rightarrow 0^+,\\
  (u_n - \bar{u})/t_n\rightarrow v\big\},
\end{multline}
\begin{multline}\nonumber
T_w(M, \bar{u}) = \big\{v\in \mathcal{U}:\ \exists \{u_n\}_{n=1}^{+\infty}\subset M,\ \{t_n\}_{n=1}^{+\infty}\subset \mathbb{R}^+,\ u_n \rightarrow \bar{u}, \ t_n\rightarrow 0^+, \\
(u_n - \bar{u})/t_n\rightharpoonup v\big\}
\end{multline}
and 
\begin{equation}\label{eq:LC}
L(\mathcal{K}, \bar{u}) = \{tv\in \mathcal{U}:\  G(\bar{u}) + G'(\bar{u})v\in \mathcal{K},\ \forall\ t > 0\}. 
\end{equation}

Let $C\subset \mathcal{U}$. The polar cone of $C$ is defined by
\begin{equation}\label{eq:dual_cone}
C^\circ = \{v\in \mathcal{U} :\ \langle v, w\rangle_{\mathcal{U}}\leq 0\quad \forall\ w\in C\}.
\end{equation}

The following property of the sequential tangent cone,  the weak sequential tangent cone and the linearizing cone is crucial to establish our theory.

\begin{lemma}\label{lem:TL_L}
If $G$ is a bounded linear operator, for any $\bar{u}\in M$, it holds
\begin{equation}\label{eq:TL_L}
  \overline{L(\mathcal{K}, \bar{u})} = T(M, \bar{u}) = T_w(M, \bar{u}).
\end{equation}
\end{lemma}

\begin{proof}
Let $v\in T(M, \bar{u})$, i.e., there exist $\{u_n\}_{n=1}^{+\infty}\subset \mathcal{U}$ and $\{t_n\}_{n=1}^{+\infty}\subset \mathbb{R}^+$ such that 
$$v = \lim\limits_{n\rightarrow +\infty}(u_n - \bar{u})/t_n,\ t_n\rightarrow 0^+,\ u_n\in M,$$
and set $v_n = u_n - \bar{u}$.
Since $G$ is linear and $u_n\in M$, we have
$$G(\bar{u}) + G'(\bar{u})v_n = G(\bar{u} + v_n) = G(u_n)\in \mathcal{K},$$
which implies $v_n/t_n \in L(\mathcal{K}, \bar{u})$ and $v\in \overline{L(\mathcal{K}, \bar{u})}$ by $v = \lim\limits_{n\rightarrow +\infty}v_n/t_n$. This leads to 
$$T(M, \bar{u}) \subset \overline{L(\mathcal{K}, \bar{u})}.$$

For any $0\neq v\in L(\mathcal{K}, \bar{u})$, there exist $v_0\in \mathcal{U}$ and $t_0 > 0$, such that $v = t_0v_0$ and
$$G(\bar{u} + v_0) =  G(\bar{u}) + G(v_0) = G(\bar{u}) + G'(\bar{u})v_0 \in \mathcal{K}.$$
Let  $u_n = \bar{u} +\frac{1}{n}v_0.$ Since $\mathcal{K}$ is convex, we have
$$ G(u_n) = G(\bar{u} +\frac{1}{n}v_0) = \frac{n-1}{n}G(\bar{u}) + \frac{1}{n}G(\bar{u} + v_0) \in \mathcal{K},$$
which implies $u_n \in M$ for any $n\in \mathbb{N}^+$.
Taking $t_n = 1/(nt_0)>0$, it follows
$$ v = \lim\limits_{n\rightarrow +\infty}(u_n - \bar{u})/t_n,$$
which gives $v\in T(M, \bar{u})$ and 
$$L(\mathcal{K}, \bar{u})\subset  T(M, \bar{u}).$$
Note that $T(M, \bar{u})$ is closed (cf. \cite[Theorem 4.10]{jahn2020introduction}). Therefore,
$$\overline{L(\mathcal{K}, \bar{u})} = T(M, \bar{u}).$$ 

Since $G$ is a bounded linear operator, $M$ is closed and convex. This further implies
$$T(M, \bar{u}) = T_w(M, \bar{u})$$
by Proposition 6.1 of  \cite{gould1975optimality}.
This completes the proof.
\end{proof}

\subsection{A first order necessary condition}
According to the classical optimization theory, at the minimizer $u^*$, the following first order necessary condition holds (cf. \cite[Theorem 1]{guignard1969generalized}, \cite[Proposition 1.2]{ito2008lagrange} or \cite[Theorem 4.14]{jahn2020introduction})
\begin{equation}\label{eq:GO_FO}
\langle D_uf(u^*), v\rangle_\mathcal{U}\geq 0\quad \forall\ v\in T_w(M,u^*),
\end{equation}
which is equivalent to 
$$ -D_uf(u^*) \in T_w^{\circ}(M,u^*).$$

We can also consider the variational inequality: Find $u^*\in M$ such that
\begin{equation}\label{eq:VI}
\langle F(u), v\rangle_\mathcal{U}\geq 0\quad \forall\ v\in T_w(M, u),
\end{equation}
where $F:\ \mathcal{U}\rightarrow  \mathcal{U}' (=  \mathcal{U}) $ is a given mapping. It is obvious that at a solution point $u^*$, there holds
$$\langle F(u^*), v\rangle_\mathcal{U}\geq 0\quad \forall\ v\in T_w(M, u^*),$$
which is in the same form of (\ref{eq:GO_FO}).
Therefore, we can deal with these two problems in the same framework, and we will only give the arguments for (\ref{eq:GO_FO}).

\subsection{The surrogate model} In this part we will give the definition of the surrogate model at a minimizer of the optimization problem (\ref{eq:GO}) and prove a fundamental theorem in our theory, i.e., Theorem \ref{thm:sm_e}. 

Note that the linearization problem  of (\ref{eq:GO})  at $u^*$ is 
\begin{equation}\nonumber
\min f(u^*) + \langle D_uf(u^*), u - u^*\rangle_{\mathcal{U}} \quad \mbox{s.t.}\ \ G(u^*) + G'(u^*)(u - u^*) \in \mathcal{K},
\end{equation}
which is equivalent to
\begin{equation}\nonumber
\min f(u^*) + \langle D_uf(u^*), u - u^*\rangle_{\mathcal{U}} \quad \mbox{s.t.}\ \  G'(u^*)u \in \mathcal{K} - G(u^*) + G'(u^*)u^*.
\end{equation}

Let us consider the following optimization problem
\begin{equation}\label{eq:sm}
\min f(u^*) + \langle D_uf(u^*), u - u^*\rangle_{\mathcal{U}} + \frac{c}{2}\|u - u^*\|_{\mathcal{U}}^2 \quad \mbox{s.t.}\ \  G'(u^*)u \in K,
\end{equation}
where $c>0$ and $ K= \mathcal{K} - G(u^*) + G'(u^*)u^*$. Note that the optimization problem (\ref{eq:sm}) is a special case of the model problem (\ref{eq:sm_g}).

The feasible set of this problem is 
$$\tilde{M} = \{u\in \mathcal{U}:\ G'(u^*)u\in K\},$$
and we have some key observations 
\begin{equation}\label{eq:KK_e}
\left\{
\begin{aligned}
K -  G'(u^*)u^* &\ \ = \ \ \mathcal{K} - G(u^*);\\
G'(u^*)u^*\in K &\Longleftrightarrow G(u^*)\in \mathcal{K}
\end{aligned}
\right.
\end{equation}
and
\begin{equation}\label{eq:so_L}
\begin{aligned}
L(K,u^*) = L(\mathcal{K}, u^*).
\end{aligned}
\end{equation}

\begin{lemma}\label{lem:so_KKT}
The classical KKT system (\ref{eq:plm}) of the optimization problem (\ref{eq:sm}) holds at $u^*$ with a proper Lagrange multiplier $\bar{\lambda}\in \mathcal{X}$, i.e., 
\begin{equation}\label{eq:sm_KKT}
\left\{
\begin{aligned}
&\langle D_uf(u^*), v\rangle_{\mathcal{U}} + \langle \bar{\lambda}, G'(u^*)v\rangle = 0\quad \forall\ v\in \mathcal{U};\\
&- \langle \bar{\lambda}, \zeta - G'(u^*)u^*\rangle \geq 0\quad \forall\ \zeta\in  K,\\
\end{aligned}
\right.
\end{equation}
if and only if the classical KKT system of the optimization problem (\ref{eq:GO}) holds at $u^*$ with the same proper Lagrange multiplier $\bar{\lambda}$, i.e.,
\begin{equation}\label{eq:GO_KKT}
\left\{
\begin{aligned}
&\langle D_uf(u^*), v\rangle_{\mathcal{U}} + \langle \bar{\lambda}, G'(u^*)v\rangle = 0\quad \forall\ v\in \mathcal{U};\\
&- \langle \bar{\lambda}, \zeta - G(u^*)\rangle \geq 0\quad \forall\ \zeta\in  \mathcal{K}.\\
\end{aligned}
\right.
\end{equation}
\end{lemma}
\begin{proof}
This follows from (\ref{eq:KK_e}) directly.
\end{proof}

Note that if $u^*$ satisfies (\ref{eq:sm_KKT}), $u^*$ is the global minimizer of the optimization problem (\ref{eq:sm}).
Meanwhile, according to the convex optimization theory, $u^*$ is the global minimizer of the optimization problem (\ref{eq:sm}) if and only if (cf.  \cite[Theorem 1]{guignard1969generalized} or \cite[Theorem 4.14 and Theorem 4.19]{jahn2020introduction})
$$\langle D_uf(u^*), v\rangle_\mathcal{U}\geq 0\quad \forall\ v\in T(\tilde{M},u^*),$$
which is equivalent to
$$\langle D_uf(u^*), v\rangle_\mathcal{U}\geq 0\quad \forall\ v\in \overline{L(\mathcal{K},u^*)}$$
by Lemma \ref{lem:TL_L} and (\ref{eq:so_L}). That is
\begin{equation}\label{eq:c_iff}
-D_uf(u^*) \in \overline{L(\mathcal{K}, u^*)}^\circ = L^\circ(\mathcal{K}, u^*).
\end{equation}

On the other hand, Lemma 4.2 of \cite{gould1975optimality} states that
\begin{equation}\label{eq:L_sub_T}
L^\circ(\mathcal{K}, u^*) \subset T_w^\circ(M,u^*).
\end{equation}
Therefore, the first order necessary condition (\ref{eq:GO_FO}) holds automatically if $u^*$ is the global minimizer of the optimization problem (\ref{eq:sm}). 

 This inspires us to use  the optimization problem (\ref{eq:sm}) to investigate the KKT system of the optimization problem (\ref{eq:GO})  at $u^*$.

\begin{Definition}\label{def:sm}
For $f\in \mathcal{F}(u^*)$, the optimization problem (\ref{eq:sm}) is called a surrogate model of the optimization problem (\ref{eq:GO}) at $u^*$ if $u^*$ is the global minimizer of the optimization problem (\ref{eq:sm}).
\end{Definition}

Theorem \ref{thm:sm_e} establishes the existence results of the surrogate model of the optimization problem (\ref{eq:GO}) at $u^*$. We recall here Guignard's condition (cf. \cite{gould1975optimality, guignard1969generalized}), which is
\begin{equation}\label{eq:GCQ}
T_w^\circ(M,u^*) = L^{\circ}(\mathcal{K}, u^*).
\end{equation}

\subsection{Proof of Theorem \ref{thm:sm_e}} 
For any $f\in \mathcal{F}(u^*)$, since $u^*$ is a global minimizer of the optimization problem (\ref{eq:sm}),  the condition (\ref{eq:c_iff}) should be held, i.e.,
$ -D_uf(u^*)\in L^\circ(\mathcal{K},u^*)$,
which yields
$$D\mathcal{F}(u^*) \subset  L^\circ(\mathcal{K},u^*).$$
Here $D\mathcal{F}(u^*) = \{-D_uf(u^*)\in \mathcal{U}:\ f\in \mathcal{F}(u^*)\}$.
Note that Theorem 3.2 of \cite{gould1975optimality} gives
$$D\mathcal{F}(u^*) = T_w^\circ(M,u^*).$$
Hence, $T_w^\circ(M,u^*) = D\mathcal{F}(u^*)  \subset L^\circ(\mathcal{K}, u^*)$.
Combining with  (\ref{eq:L_sub_T}), we arrive at (\ref{eq:GCQ}).

For the other direction, (\ref{eq:GO_FO}) holds, i.e.,
$$-D_uf(u^*)\in T_w^\circ(M,u^*)$$
 for any $f\in \mathcal{F}(u^*)$. If (\ref{eq:GCQ}) holds, we have (\ref{eq:c_iff}) holds, which implies $u^*$ is a global minimizer of the optimization problem (\ref{eq:sm}). Hence, the optimization problem (\ref{eq:sm}) is a surrogate model of the optimization problem (\ref{eq:GO}) at $u^*$. This completes the proof.

\begin{Remark}
Theorem \ref{thm:sm_e} actually gives the mathematical foundation of quadratic programming(QP)-based methods. If Guignard's condition (\ref{eq:GCQ}) fails, there exists an objective function $f\in \mathcal{F}(u^*)$ such that the minimizer $u^*$ of the original problem (\ref{eq:GO}) can not be obtained by solving quadratic programming problems. In other words, the QP-based methods fail in this case.
\end{Remark}

\begin{Remark}
Theorem \ref{thm:sm_e} also indicates that for non-convex optimization problems, if ones add any assumptions implying Guignard's condition (\ref{eq:GCQ}) (which is usually the case in the literature) and do the local theoretical analysis, they actually do the analysis for some convex optimization problems. In this sense, most of the existing theoretical results for non-convex optimization problems are essentially the results for convex optimization problems.
\end{Remark}

\subsection{Nonsmooth cases} In the previous arguments, we assume that $f$ and $G$ are Fr\'{e}chet differentiable. It is worth noting that the theory in this paper can also be applied to some nonsmooth cases, for example the case where $f$ and $G$ are only semismooth. In this case, we can choose $p_{u^*}\in \partial f(u^*)$ and $S_{u^*}\in \partial G(u^*)$ to do the analysis. If there exist $p_{u^*}\in \partial f(u^*)$ and $S_{u^*}\in \partial G(u^*)$ such that
\begin{enumerate}[(1)]
\item $\langle p_{u^*}, v\rangle_\mathcal{U}\geq 0\quad \forall\ v\in T_w(M,u^*)$;

\item $T_w^\circ(M,u^*) = L^\circ(\mathcal{K}, u^*, S_{u^*})$, where $L(\mathcal{K}, u^*, S_{u^*})$ is $L(\mathcal{K}, u^*)$ with $G'(u^*) = S_u^*$,
\end{enumerate}
then the surrogate model can be defined as
\begin{equation}\nonumber
\min f(u^*) + \langle p_{u^*}, u - u^*\rangle_{\mathcal{U}} + \frac{c}{2}\|u - u^*\|_{\mathcal{U}}^2 \quad \mbox{s.t.}\ \  S_{u^*}u \in K = \mathcal{K} - G(u^*) + S_{u^*}u^*
\end{equation}
for any $c>0$.

\section{The essential Lagrange multiplier}\label{sec:ELM}
\setcounter{equation}{0}
In this section we will establish the basic theory of the essential Lagrange multiplier and the proper Lagrange multiplier of the model problem (\ref{eq:sm_g}). The existence of the essential Lagrange multiplier shows essential differences in finite and infinite-dimensional spaces. As an application of the essential Lagrange multiplier, we will use it to characterize the convergence of the multipliers generated by the classical ALM (see (\ref{eq:ALM})). 

Without further explanation, we always assume that $u^*$ is the global minimizer of the model problem (\ref{eq:sm_g}) in this section. 

\subsection{The essential Lagrange multiplier}
We will establish the existence and uniqueness theory of the essential Lagrange multiplier (see Definition \ref{def:elm}) here. Our results indicate the existence theory of the essential Lagrange multiplier is different in finite and infinite-dimensional cases. More precisely, the essential Lagrange multiplier always exists in the finite-dimensional case  (see Corollary \ref{cor:finite-d}), while in the infinite-dimensional case, this is no longer true (cf. \cite{hinze2008optimization, ito2008lagrange}). This can not be achieved by the existing optimization theory based on separation theorems for the proper Lagrange multiplier, because the interior point condition is essential for separation theorems in infinite-dimensional spaces (cf. \cite[Chapter 1, Remark 4]{brezis2011functional}).

\begin{theorem}\label{thm:ELM_unique}
The essential Lagrange multiplier is unique.
\end{theorem}
\begin{proof}
This can be derived from Definition \ref{def:elm} directly. 
\end{proof}

\begin{theorem}\label{thm:ELM_exist_iff}
The essential Lagrange multiplier exists at the global minimizer $u^*$ of the model problem (\ref{eq:sm_g}) if and only if 
$$ -D_u\theta(u^*) \in R(S^*),$$
where $S^*$ is the adjoint operator of $S$.
\end{theorem}
\begin{proof}
If the essential Lagrange multiplier $\lambda^*$ exists, according to its definition, we have
$$ -D_u\theta(u^*) = S^*\lambda^*,$$
which means $-D_u\theta(u^*) \in R(S^*)$.

If $-D_u\theta(u^*) \in R(S^*)$, there exists $\bar{\lambda}^*\in \mathcal{X}$ such that
$$-D_u\theta(u^*) = S^*\bar{\lambda}^*,$$ 
which is equivalent to
\begin{equation}\label{eq:elm-iff1}
\langle D_u\theta(u^*), v\rangle_\mathcal{U} + \langle \bar{\lambda}^*, Sv\rangle = 0\quad \forall\ v\in \mathcal{U}.
\end{equation}
Since $u^*$ is the global minimizer, by Theorem \ref{thm:w-akkt}, there exists $\{\lambda^k\}_{k=1}^{+\infty}\subset \mathcal{X}$ such that the W-AKKT system (\ref{eq:w-akkt}) holds. Therefore,  
$$\lim\limits_{k\rightarrow + \infty}\langle \lambda^k, Sv\rangle = \langle \bar{\lambda}^*, Sv\rangle\quad \forall\ v\in \mathcal{U},$$
by (\ref{eq:elm-iff1}), and then
$$ - \langle \bar{\lambda}^*, Sv - Su^* \rangle =  -\lim\limits_{k\rightarrow + \infty}\langle \lambda^k, Sv - Su^*\rangle \geq -  \uplim\limits_{k\rightarrow + \infty}\langle \lambda^k, Sv - Su^*\rangle \geq 0\quad \forall\ v\in \mathcal{U},$$
by (\ref{eq:w-akkt}). Finally, by the boundedness of $\bar{\lambda}^*$, we arrive at
$$ - \langle \bar{\lambda}^*, \zeta - Su^* \rangle \geq 0\quad \forall\ \zeta\in \overline{K\cap R(S)},$$
which, together with (\ref{eq:elm-iff1}), implies that the restriction of $\bar{\lambda}^*$ to $\overline{R(S)}$ is the essential Lagrange multiplier. This gives the existence of the essential Lagrange multiplier.
\end{proof}

\begin{theorem}\label{thm:ELM_exist_aiff}
If $R(S)$ is closed in $\mathcal{X}$, then the essential Lagrange multiplier exists. Conversely,  if the essential Lagrange multiplier always exists at $u^*$  for any $K$ and $\theta(\cdot)$ satisfying the assumptions of the model problem (\ref{eq:sm_g}), then $R(S)$ is closed in $\mathcal{X}$. 
\end{theorem}
\begin{proof}
According to Theorem \ref{thm:w-akkt}, there exists $\{\lambda^k\}_{k=1}^{+\infty}\subset \mathcal{X}$ such that
$$ \lim\limits_{k\rightarrow +\infty}\langle S^*\lambda^k + D_u\theta(u^*), v \rangle_\mathcal{U} = 0\quad \forall\ v\in \mathcal{U},$$
i.e., $\{S^*\lambda^k\}_{k=1}^{+\infty}$ weakly converges to $-D_u\theta(u^*)$ in $\mathcal{U}$. 

If $R(S)$ is closed, by the closed range theorem (cf. \cite[p. 205]{yosida2012functional}), $R(S^*)$ is closed, which further is weakly closed. Therefore, $-D_u\theta(u^*)\in R(S^*)$ and then by Theorem \ref{thm:ELM_exist_iff}, we have the existence of the essential Lagrange multiplier.

Let $u_0\in \mathrm{Ker}(S)^\perp$ be arbitrary, $\theta(u) =  \frac{1}{2}\|u\|_{\mathcal{U}}^2$ and  $ K = \{Su_0\}$. In this case the global minimizer $u^* = u_0$ and $D_u\theta(u^*) = u_0\in \mathrm{Ker}(S)^\perp$. Note that $\mathrm{Ker}(S)^\perp=\overline{R(S^*)}$. Since $u_0$ is arbitrary, by the assumption on the existence of the essential Lagrange multiplier and Theorem  \ref{thm:ELM_exist_iff}, we have $\mathrm{Ker}(S)^\perp\subset R(S^*)$, which gives $\overline{R(S^*)} \subset R(S^*)$. Therefore, $R(S^*) = \overline{R(S^*)}$, which is equivalent to $R(S) = \overline{R(S)}$ by the closed range theorem, i.e., $R(S)$ is closed. 
\end{proof}

\begin{Remark}
Theorem \ref{thm:ELM_exist_aiff} implies that the condition that $R(S)$ is closed in $\mathcal{X}$ is sufficient and almost necessary for the existence of the essential Lagrange multiplier.
\end{Remark}

Note that if $R(S)$ is a finite-dimensional space, $R(S)$ is closed. We have the following corollaries.
\begin{Corollary}\label{cor:finite}
If $R(S)$ is a finite-dimensional space, then the essential Lagrange multiplier exists.
\end{Corollary}
\begin{Corollary}\label{cor:finite-d}
If $\mathcal{X}$ is a finite-dimensional space, then the essential Lagrange multiplier exists.
\end{Corollary}

The essential Lagrange multiplier can also be used to characterize the optimality of a feasible point.

\begin{theorem}\label{thm:optimality}
If the essential Lagrange multiplier exists at a feasible point $u^*$, then $u^*$ is the global minimizer of the model problem (\ref{eq:sm_g}).
\end{theorem}
\begin{proof}
According to the convex optimization theory, it suffices to show that (cf. \cite[Proposition 26.5]{bauschke2011convex}, \cite[Proposition 2.1, Chapter II]{ekeland1999convex} or \cite[Corollary 4.20]{jahn2020introduction}) 
$$ \langle D_u\theta(u^*), v - u^*\rangle_\mathcal{U} \geq 0\quad\ \forall\ Sv\in K.$$
Let $\lambda^*$ be the essential Lagrange multiplier. The definition of $\lambda^*$ gives
$$
\langle D_u\theta(u^*), v - u^*\rangle_\mathcal{U} 
 = -\langle \lambda^{*}, S(v - u^*)\rangle 
 =  -\langle \lambda^{*}, Sv - Su^*\rangle
 \geq 0\quad \forall\ Sv\in K,
 $$
 which completes the proof.
\end{proof}

Furthermore, Theorem \ref{thm:ELM_exist_aiff} and Theorem \ref{thm:optimality} lead to the following corollary.
\begin{Corollary}
Assume that $R(S)$ is closed in $\mathcal{X}$, a feasible point $u^*$ is the global minimizer of the model problem (\ref{eq:sm_g}) if and only if the essential Lagrange multiplier exists at $u^*$.
\end{Corollary}

\subsection{The proper Lagrange multiplier} According to the definition of the proper Lagrange multiplier (see Definition \ref{def:plm}) and the definition of the essential Lagrange multiplier (see Definition \ref{def:elm}), the existence of the proper Lagrange multiplier always implies the existence of the essential Lagrange multiplier and 
\begin{equation}\label{eq:plm-elm}
\lambda^* = \bar{\lambda}|_{\overline{R(S)}},
\end{equation}
where $\bar{\lambda}|_{\overline{R(S)}}$ is the restriction of $\bar{\lambda}$ to $\overline{R(S)}$. In other words, if the essential Lagrange multiplier does not exist, neither does the proper Lagrange multiplier. Therefore, we can consider the theory of the proper Lagrange multiplier under the assumption that the essential Lagrange multiplier exists, and which can be verified by the results in the previous subsection. We can also assume that $\lambda^* \neq 0$. Otherwise $\bar{\lambda} = 0$ is a proper Lagrange multiplier.  We will establish the existence and uniqueness theory of the proper Lagrange multiplier under Assumption \ref{assum:e_elm}. 

\begin{Assumption}\label{assum:e_elm}
The essential Lagrange multiplier $\lambda^*$ exists at $u^*$ and $\lambda^* \neq 0$.
\end{Assumption}

\begin{Remark}
If the essential Lagrange multiplier does not exist, neither does the classical KKT system. This happens in infinite-dimensional cases as shown in the previous subsection. Therefore, it is necessary to use the asymptotic or approximate KKT system to characterize the optimality in some infinite-dimensional cases, which has been used in the literature (cf. \cite{andreani2019sequential, borgens2020new, jeyakumar2003new, steck2018lagrange, thibault1997sequential}) as a technique, but did not confirm its necessity theoretically. 
\end{Remark}

Let $\zeta^* = Su^*$ and $\mathcal{N}(\zeta^*,K)$ be the normal cone to $K$ at $\zeta^*$, i.e., 
$$\mathcal{N}(\zeta^*,K) = \{\lambda\in \mathcal{X}:\ -\langle \lambda, \zeta - \zeta^*\rangle \geq 0\quad \forall\ \zeta \in K\}.$$

\begin{theorem}\label{thm:pLM_exist}
Suppose that Assumption \ref{assum:e_elm} holds. The proper Lagrange multiplier exists at $u^*$ if and only if there exist  $\tilde{\lambda}\in \mathcal{N}(\zeta^*,K)$ and $\bar{\zeta}_0\in \overline{R(S)}$ such that
\begin{equation}\label{eq:LM_iff1}
 \mathrm{Ker}(\tilde{\lambda})\cap \overline{R(S)}=  \mathrm{Ker}(\lambda^*)\cap \overline{R(S)}\quad\quad (\mbox{Compatibility\ Condition})
 \end{equation}
 or equivalently
\begin{equation}\nonumber
\mathrm{Span}\{\tilde{\lambda}\} + \mathrm{Ker}(S^*) = \mathrm{Span}\{\lambda^*\} + \mathrm{Ker}(S^*)
 \end{equation}
and 
 \begin{equation}\label{eq:LM_iff2}
 \left\{
 \begin{aligned}
 & \langle\tilde{\lambda}, \bar{\zeta}_0 \rangle > 0,\\
 & \langle \lambda^*, \bar{\zeta}_0 \rangle > 0. 
 \end{aligned}
 \right.
 \quad\quad  (\mbox{Consistency\ Condition})
 \end{equation}
\end{theorem}

\begin{proof}
A direct calculation of the polar sets on both sides gives the equivalence.

If the proper Lagrange multiplier $\bar{\lambda}$ exists, we have $\bar{\lambda} \in \mathcal{N}(\zeta^*,K)$ and $\lambda^* = \bar{\lambda}|_{\overline{R(S)}}$, which follows 
$$ \mathrm{Ker}(\bar{\lambda})\cap \overline{R(S)} = \mathrm{Ker}(\lambda^*)\cap \overline{R(S)}.$$
Since $\lambda^*\neq 0$ on $\overline{R(S)}$, there exists $\zeta_0\in \overline{R(S)}$ such that
$$\langle \lambda^*, \zeta_0\rangle \neq 0.$$
We assume that
$\langle \lambda^*, \zeta_0\rangle > 0.$
Otherwise, we can choose $-\zeta_0$. Taking $\bar{\zeta}_0 = \zeta_0$, the condition (\ref{eq:LM_iff2}) holds for $\bar{\lambda}$. 
Hence, we can take $\tilde{\lambda} = \bar{\lambda}$. 

Now we prove the other direction. 
Let $\tilde{\lambda}$ be an element in $\mathcal{N}(\zeta^*,K)$ such that (\ref{eq:LM_iff1}) and (\ref{eq:LM_iff2}) hold. 

Let  
$$ \bar{\lambda} = t_0\tilde{\lambda},$$
where $t_0 = \langle \lambda^*, \bar{\zeta}_0\rangle/\langle \tilde{\lambda}, \bar{\zeta}_0\rangle$ and $\bar{\zeta}_0$ satisfies (\ref{eq:LM_iff2}). 

We will prove that $\bar{\lambda}$ is a proper Lagrange multiplier. Note that the  condition (\ref{eq:LM_iff2}) implies $t_0>0$, which further gives $\bar{\lambda}\in \mathcal{N}(\zeta^*,K)$. 

Now, by the definition of the proper Lagrange multiplier,  we only need to show that
$$\lambda^* = \bar{\lambda}|_{\overline{R(S)}}.$$

By  (\ref{eq:LM_iff2}) and (\ref{eq:LM_iff1}), we have $\overline{R(S)} = \mathrm{Ker}(\lambda^*)\cap \overline{R(S)} + \mathrm{Span}\{\bar{\zeta}_0\} = \mathrm{Ker}(\tilde{\lambda})\cap \overline{R(S)} + \mathrm{Span}\{\bar{\zeta}_0\}$. Therefore, for any $\zeta \in \overline{R(S)}$, there exist $s\in \mathbb{R}$ and $\zeta_0\in \mathrm{Ker}(\tilde{\lambda})\cap \overline{R(S)}$ such that $\zeta = \zeta_0 + s\bar{\zeta}_0$. It follows
$$\langle \bar{\lambda}, \zeta\rangle =  t_0 \langle \tilde{\lambda}, \zeta\rangle =  t_0 \langle \tilde{\lambda}, \zeta_0 + s\bar{\zeta}_0\rangle = t_0s\langle \tilde{\lambda}, \bar{\zeta}_0\rangle = s\langle \lambda^*, \bar{\zeta}_0\rangle =  \langle \lambda^*, \zeta_0 + s\bar{\zeta}_0\rangle = \langle \lambda^*, \zeta\rangle,$$
which implies $\lambda^* = \bar{\lambda}|_{\overline{R(S)}}.$
\end{proof}

\begin{theorem}\label{thm:pLM_unique}
Suppose that Assumption \ref{assum:e_elm} holds.
There exists a unique \\
proper Lagrange multiplier at $u^*$ if and only if (1) there exists $\tilde{\lambda}\in\mathcal{N}(\zeta^*,K)$ which satisfies (\ref{eq:LM_iff1}) and (\ref{eq:LM_iff2}), and (2) for any $\hat{\lambda}\in\mathcal{N}(\zeta^*,K)$ which satisfies (\ref{eq:LM_iff1}) and (\ref{eq:LM_iff2}),  it holds $\mathrm{Ker}(\hat{\lambda}) = \mathrm{Ker}(\tilde{\lambda})$.
\end{theorem}

\begin{proof}
Since all the proper Lagrange multipliers belong to $\mathcal{N}(\zeta^*,K)$ and their restrictions to $\overline{R(S)}\ (\neq \{0\})$ are the same, according to the proof of Theorem \ref{thm:pLM_exist}, we have $\bar{\lambda} = \langle \lambda^*, \bar{\zeta}_0\rangle/\langle \tilde{\lambda}, \bar{\zeta}_0\rangle\tilde{\lambda}$ is the unique proper Lagrange multiplier.

Let $\bar{\lambda}$ be the unique proper Lagrange multiplier. According to the condition (\ref{eq:LM_iff2}) of Theorem \ref{thm:pLM_exist}, $\bar{\lambda} \in \mathcal{N}(\zeta^*,K)$, and (\ref{eq:LM_iff1}) and (\ref{eq:LM_iff2}) hold for $\bar{\lambda}$. 

Suppose that there exists $\tilde{\lambda}\in\mathcal{N}(\zeta^*,K)$ which satisfies (\ref{eq:LM_iff1}) and (\ref{eq:LM_iff2}), it holds $\mathrm{Ker}(\tilde{\lambda}) \neq \mathrm{Ker}(\bar{\lambda})$. By Theorem \ref{thm:pLM_exist}, there exists a proper Lagrange multiplier $\bar{\lambda}_0$ with $\mathrm{Ker}(\bar{\lambda}_0)\neq \mathrm{Ker}(\bar{\lambda})$. This is a contradiction to the uniqueness of the Lagrange multiplier. 
\end{proof}

\begin{Remark}
If $\mathcal{X} = \overline{R(S)}$ or equivalently $\mathrm{Ker}(S^*) = \{0\}$, the proper Lagrange multiplier is the essential Lagrange multiplier, which implies that the proper Lagrange multiplier (if exists) is unique.  
\end{Remark}

\begin{Remark}
In the finite-dimensional case, the condition $\mathrm{Ker}(S^*) = \{0\}$ is the LICQ condition. The connections of the LICQ condition and the uniqueness of the proper Lagrange multiplier have been investigated in \cite{wachsmuth2013licq}. It has also been proved that the proper Lagrange multiplier is unique if and only if it satisfies SMFC in \cite{kyparisis1985uniqueness} for an  optimization problem with both equality and inequality constraints under the assumption that the proper Lagrange multiplier exists. The uniqueness results for the case of general cone constraints can be found in \cite{shapiro1992perturbation, shapiro1997uniqueness}.
\end{Remark}

The following example is helpful in understanding the previous theoretical results.
\begin{Example}\label{exm:Ex1}
We consider the optimization problem
\begin{equation}\nonumber
\min\frac{1}{2}[(x_1 -\alpha)^2 + x_2^2]\ \ \mbox{s.t.}\ \  Sx \in K,
\end{equation}
where $\alpha\in \mathbb{R}$, $x\in \mathbb{R}^{2}$, $S = \begin{pmatrix}1 & 0\\ 0 &0 \end{pmatrix}$ and $K\subset \mathbb{R}^2$ is a closed convex set.

\begin{enumerate}[$\bullet$]
\item Let $K = K_1$, where 
$$K_1 = \{(\zeta_1,\zeta_2)^T\in \mathbb{R}^2:\ \zeta_1^2 + (\zeta_2 -1)^2 \leq 1\}.$$
In this case $K\cap R(S) = (0,0)^T$ and the feasible set is 
$$ M = \{(0,x_2)^T\in \mathbb{R}^2:\ x_2\in \mathbb{R}\}. $$ 
The global minimizer of this problem is $x^* = (0,0)^T$. The gradient of the objective function at this point is $(-\alpha,0)^T$ and the essential Lagrange multiplier is $\lambda^* = (\alpha, 0)^T$. Note that the solution of
$$\begin{pmatrix}-\alpha \\0 \end{pmatrix} + S^*\lambda = 0$$
is $\lambda = (\alpha, \lambda_2)^T$, for any $\lambda_2\in \mathbb{R}$. Since the proper Lagrange multiplier must satisfy the above equation, we assume that $\bar{\lambda} = (\alpha, \bar{\lambda}_2)^T$ for some $\bar{\lambda}_2\in \mathbb{R}$. Now we consider the condition
$$- \langle \bar{\lambda}, \zeta - \zeta^*\rangle\geq 0\quad \forall\ \zeta\in K,$$
where $\zeta^* = Sx^* = (0,0)^T$.
It is equivalent to 
$$ -\alpha\zeta_1 - \bar{\lambda}_2 \zeta_2 \geq 0\quad \forall\ \zeta = (\zeta_1, \zeta_2)^T\in K.$$
If $\alpha = 0$, we can choose $\bar{\lambda}_2 = 0$.
If $\alpha \neq 0$, there is no $\bar{\lambda}_2\in \mathbb{R}$ to satisfy the inequality.  This means that unless $\alpha=0$, the proper Lagrange multiplier does not exist for this problem at the global minimizer. Note that if $\alpha \neq 0$, we have 
$$ \mathcal{N}(\zeta^*,K):=\{(0, \lambda_2)^T\in \mathbb{R}^2:\ \lambda_2\leq 0\} \subset \mathrm{Ker}(S^*) :=\{(0, \lambda_2)^T\in \mathbb{R}^2:\ \lambda_2\in \mathbb{R}\}.$$
Both the compatibility condition (\ref{eq:LM_iff1}) and the consistency condition (\ref{eq:LM_iff2}) are not satisfied.

\item Let $K = K_2$, where
$$
K_2 = \{(\zeta_1,\zeta_2)^T\in \mathbb{R}^2:\ \zeta_1^2 + (\zeta_2 -1)^2 \leq 1\}\setminus \{(\zeta_1,\zeta_2)^T\in \mathbb{R}^2:\ \zeta_1 - \zeta_2\leq 0\}.
$$
The feasible set $M$, the global minimizer and the essential Lagrange multiplier are the same as those in the case $K = K_1$. As before, for the proper Lagrange multiplier $\bar{\lambda} = (\alpha, \bar{\lambda}_2)^T$ we have
$$ -\alpha\zeta_1 - \bar{\lambda}_2 \zeta_2 \geq 0\quad \forall\ \zeta = (\zeta_1, \zeta_2)^T\in K.$$
On the other hand, the normal cone $\mathcal{N}(\zeta^*, K)$ is given by
$$\mathcal{N}(\zeta^*, K) = \{(\lambda_1, \lambda_2)^T\in \mathbb{R}^2:\ \lambda_2\leq 0,\ \lambda_1 + \lambda_2 \geq 0\}.$$

\begin{enumerate}[(i)]
\item If $\alpha<0$, $(\alpha, \bar{\lambda}_2)^T\not\in \mathcal{N}(\zeta^*, K)$ for any $\bar{\lambda}_2\in \mathbb{R}$. Therefore, the proper Lagrange multiplier does not exist. Note that the condition (\ref{eq:LM_iff2}) can not be satisfied, while the condition (\ref{eq:LM_iff1}) is always true for any $0\neq \lambda \in \mathcal{N}(\zeta^*, K)$.

\item If $\alpha>0$, $\bar{\lambda}$ is a proper Lagrange multiplier for any $\bar{\lambda}_2\in [-\alpha, 0]$. In this case, both the compatibility condition (\ref{eq:LM_iff1}) and the consistency condition (\ref{eq:LM_iff2}) can be fulfilled. 
\end{enumerate}

\end{enumerate}

Figure 1 gives an illustration of the results above. 

\begin{center}
\begin{tikzpicture}[scale = 0.5]
\begin{scope}[scale = 1.3]
\draw [color = black, fill=grey!20] (1.2,1.2) arc (0:360:1.2);
\draw[->, blue] (-1.5,0)--(1.8,0) node[right]{{\color{black}$\zeta_1$}};
\draw[->] (0,-0.5)--(0,3.0) node[above]{$\zeta_2$};
\node at (0,-1.0) {No $\bar{\lambda}$ for $\alpha\neq 0$};
\draw[fill=black!80] (0,1.2) circle (0.04) node[right]{$K_1$};
\draw[fill=red!80] (0,0) circle (0.06);
\node at (-2.1,0) {\color{blue}{\footnotesize{$R(S)$}}};
\end{scope}


\begin{scope}[xshift = 9cm, scale = 1.3]
\draw [color = black, fill=grey!20] (1.2,1.2) arc (0:270:1.2);
\draw [color = black] (0,0)--(1.2,1.2);
\draw[->, blue] (-1.5,0)--(1.8,0) node[right]{{\color{black}$\zeta_1$}};
\draw[->] (0,-0.5)--(0,3.0) node[above]{$\zeta_2$};
\node at (0,-1.0) {No $\bar{\lambda}$ for $\alpha<0$};
\node at (-2.1,0) {\color{blue}{\footnotesize{$R(S)$}}};
\draw[fill=black!80] (0,1.2) circle (0.04) node[right]{$K_2$};
\draw[fill=red!80] (0,0) circle (0.06);
\end{scope}

\begin{scope}[xshift = 18cm, scale = 1.3]
\draw [color = black, fill=grey!20] (1.2,1.2) arc (0:270:1.2);
\draw [color = black] (0,0)--(1.2,1.2);
\draw[->, blue] (-1.5,0)--(1.8,0) node[right]{{\color{black}$\zeta_1$}};
\draw[->] (0,-0.5)--(0,3.0) node[above]{$\zeta_2$};
\node at (0.5,-1.0) {$\bar{\lambda} = (\alpha, \bar{\lambda}_2)^T$ for $\alpha\geq 0$};
\draw[fill=black!80] (0,1.2) circle (0.04) node[right]{$K_2$};
\draw[fill=red!80] (0,0) circle (0.06);
\node at (-2.1,0) {\color{blue}{\footnotesize{$R(S)$}}};
\end{scope}
\node at (9.5,-3.5) {$\mathrm{Figure\ 1.}$};
\end{tikzpicture}
\end{center}
\end{Example}

\subsection{The convergence of multipliers of the classical ALM} The classical ALM for the model problem is given in (\ref{eq:ALM}). We will give an essential characterization of the convergence of the multipliers generated by the algorithm.

\begin{theorem}\label{thm:ALM-ELM}
Let $\{\lambda^k\}_{k=1}^{+\infty}$ be the multipliers generated in the classical ALM (\ref{eq:ALM}). 
\begin{enumerate}[(1)]
\item If the essential Lagrange multiplier $\lambda^*$ exists at $u^*$, then 
\begin{equation}\label{eq:m-c}
\lim\limits_{k\rightarrow + \infty} \langle \lambda^k, Sv\rangle = \langle \lambda^*, Sv\rangle\quad \forall\ v\in \mathcal{U}.
\end{equation}

\item If the restriction of $\{\lambda^k\}_{k=1}^{+\infty}$ to $\overline{R(S)}$ weakly converges in $\overline{R(S)}$ to some element $\lambda^*\in \overline{R(S)}$, then $\lambda^*$ is the essential Lagrange multiplier at $u^*$.
\end{enumerate}

\end{theorem}

\begin{proof}
Note that  $\{\lambda^k\}_{k=1}^{+\infty}$ satisfies (\ref{eq:w-akkt}).  The first part follows from (\ref{eq:w-akkt}) and Definition \ref{def:elm} directly.

For the second part, if the restriction of $\{\lambda^k\}_{k=1}^{+\infty}$ to $\overline{R(S)}$ weakly converges in $\overline{R(S)}$ to some element $\lambda^*\in \overline{R(S)}$, i.e., 
$$\lim\limits_{k\rightarrow +\infty}\langle \lambda^k, Sv\rangle =  \langle \lambda^*, Sv\rangle\quad \forall\ v\in \mathcal{U},$$
then we have
$$ \langle D_u\theta(u^*), v\rangle_\mathcal{U} +   \langle \lambda^*, Sv\rangle = \langle D_u\theta(u^*), v\rangle_\mathcal{U}  + \lim\limits_{k\rightarrow +\infty}\langle \lambda^{k}, Sv\rangle = 0\quad \forall\ v\in \mathcal{U}$$
and
$$ -\langle \lambda^*, Sv - Su^*\rangle = -\lim\limits_{k\rightarrow +\infty} \langle \lambda^k, Sv - Su^*\rangle\geq -\uplim\limits_{k\rightarrow +\infty} \langle \lambda^k, Sv - Su^*\rangle\geq 0\quad \forall\ Sv\in K,$$
by (\ref{eq:w-akkt}). Therefore, $\lambda^*$ is the essential Lagrange multiplier by the boundedness of $\lambda^*$ and we finish the proof.

\end{proof}

The following corollary follows from Theorem \ref{thm:ELM_exist_aiff} and Theorem \ref{thm:ALM-ELM} directly.
\begin{Corollary}
If $R(S)$ is closed in $\mathcal{X}$, then the restriction of $\{\lambda^k\}_{k=1}^{+\infty}$ to $R(S)$ always weakly converges in $R(S)$ to the essential Lagrange multiplier $\lambda^*$ at $u^*$.
\end{Corollary}

\begin{Example}\label{Exm:ALM}
Let us consider the optimization problem
\begin{equation}\nonumber
\min\frac{1}{2}x^2\quad \mbox{s.t.}\ 
\left\{ 
\begin{aligned}
x & = 1;\\
2x & = 2.
\end{aligned}
\right.
\end{equation}
The global minimizer of this problem is $x^* = 1$ and the essential Lagrange multiplier at $x^*$ is 
$$\lambda^* = -\frac{1}{5}\begin{pmatrix}1\\ 2 \end{pmatrix}.$$
All the proper Lagrange multipliers of this example define the set
$$\Lambda = \{\bar{\lambda} = (\lambda_1, \lambda_2)^T\in \mathbb{R}^2:\ \lambda_1 + 2\lambda_2 + 1 = 0\}.$$

The iterators of the classical ALM satisfy 
\begin{equation}\nonumber
\left\{
\begin{aligned}
x^{k+1} - 1 & =  \frac{1}{1+5\beta}(x^{k} - 1);\\
\lambda^{k+1} + \frac{1}{5}\begin{pmatrix}1\\ 2 \end{pmatrix}
& = J\left[\lambda^k + \frac{1}{5}\begin{pmatrix}1\\ 2 \end{pmatrix}\right];\\
\lambda^{k+1}_1 + 2\lambda^{k+1}_2 +1 &= \frac{1}{1+5\beta}(\lambda^{k} + 2\lambda^{k} + 1),
\end{aligned}
\right.
\end{equation}
where $k = 1, 2, \dots$,
$$ J = \frac{1}{1+5\beta}\begin{pmatrix}1+4\beta &  -2\beta\\ -2\beta & 1 + \beta \end{pmatrix}$$
and $\beta > 0$.

Note that $J$ is symmetric, positive definite and $\det J = 1$, which implies $\rho(J) > 1$. We have the following convergence results.
\begin{itemize}
\item $\{x^k\}_{k=1}^{+\infty}$ converges to $x^* = 1$.
\par\vspace{4pt}

\item $\{\lambda^k\}_{k=1}^{+\infty}$ is unbounded.
\par\vspace{4pt}

\item $\{\lambda^k\}_{k=1}^{+\infty}$ converges to the essential Lagrange multiplier in $R(S)$, i.e., 
$$\lambda^k_1 + 2\lambda^k_2 \rightarrow -1 = -\frac{1}{5}(1\times1 + 2\times2)\ \ \mbox{as}\ \ k\rightarrow + \infty.$$

\item The distance between $\lambda^k\ (k=1, 2, \dots)$ and $\Lambda$ converges to $0$.

\end{itemize}

\end{Example}

\section{Revisiting the weak form asymptotic KKT system}\label{sec:w-akkt}
\setcounter{equation}{0}
In the previous arguments we notice that in the infinite-dimensional case, the essential Lagrange multiplier may not exist. Here we derive a sufficient condition to guarantee the existence of the essential Lagrange multiplier by the W-AKKT system (\ref{eq:w-akkt}). We will also use the W-AKKT system (\ref{eq:w-akkt}) to give a proof of the existence of the proper Lagrange multiplier under a generalized condition of Robinson's condition  (cf. \cite[p. 5]{ito2008lagrange}, \cite{maurer1979first, robinson1976stability, zowe1979regularity} or (\ref{eq:RCQ})).  It is worth noting that all these results also give the convergence results of the multipliers generated by the classical ALM for the model problem. 

\begin{theorem}\label{thm:Rcl}
If there exists  a norm  $\|\cdot\|_*$ on $R(S)$ such that $(R(S), \|\cdot\|_*)$ is a Banach space which is continuously embedded into $(\mathcal{X}, \|\cdot\|)$, then there exists a unique essential Lagrange multiplier $\tilde{\lambda}^*\in (R(S), \|\cdot\|_*)'$ which is the dual space of $(R(S), \|\cdot\|_*)$ such that the following KKT system holds, i.e.,
\begin{equation}\label{eq:eKKT_wB}
\left\{
\begin{aligned}
& \langle D_u\theta(u^*), v\rangle_\mathcal{U}  + \langle \tilde{\lambda}^*, Sv\rangle_{*} = 0\quad \forall\ v\in \mathcal{U};\\
& - \langle\tilde{\lambda}^* ,\zeta - Su^*\rangle_*\geq 0\quad \forall\ \zeta\in \overline{K\cap R(S)}^{\|\cdot\|_*}, \\
\end{aligned}
\right.
\end{equation}
where $\langle\cdot,\cdot\rangle_*$ is the duality pair of $(R(S), \|\cdot\|_*)'$ and $(R(S), \|\cdot\|_*)$. 
\end{theorem}

\begin{proof}
Since $(R(S), \|\cdot\|_*)$ is continuously embedded into $(\mathcal{X}, \|\cdot\|)$ and $\{\lambda^k\}_{k=1}^{+\infty}\subseteq \mathcal{X}$,  $\{\lambda^k\}_{k=1}^{+\infty}\subseteq (R(S), \|\cdot\|_*)'$. By the first condition in the W-AKKT system (\ref{eq:w-akkt}), we know
$$ \sup\{|\langle \lambda^k, \zeta\rangle_{*}|:\ k=1,\dots, +\infty\}< +\infty\quad \forall\ \zeta \in R(S).$$
According to the uniform boundedness principle (cf. \cite{sokal2011really}), we know that $\{\lambda^k\}_{k=1}^{+\infty}$ is bounded in $(R(S), \|\cdot\|_*)'$. Note that the unit ball in $(R(S), \|\cdot\|_*)'$ is weak-* compact. Again, by the first condition in the W-AKKT system (\ref{eq:w-akkt}), we know that there exists $\tilde{\lambda}^*\in (R(S), \|\cdot\|_*)'$ such that (\ref{eq:eKKT_wB}) holds. The uniqueness follows from (\ref{eq:eKKT_wB}) directly.
\end{proof}

The existence of the proper Lagrange multiplier can be guaranteed by Robinson's condition (cf. \cite[p. 5]{ito2008lagrange} and \cite{maurer1979first, robinson1976stability, zowe1979regularity}), which is
\begin{equation}\label{eq:RCQ}
\mathcal{X} =  S(\mathcal{U}) - \mathcal{K}_{K,\zeta^*}, 
\end{equation}
where 
$$\mathcal{K}_{K,\zeta^*} = \{t(\zeta - \zeta^*):\ \forall\ \zeta \in K\ \ \mbox{and}\ \ \forall\ t\geq 0\}. $$

Now we give an elementary proof of the existence of proper Lagrange multipliers under a general condition (which is a generalization of the condition (\ref{eq:RCQ})) based on the W-AKKT system (\ref{eq:w-akkt}) and the uniform boundedness principle that has an elementary proof in \cite{sokal2011really}.

\begin{lemma}\label{lem:wakkt-eq}
The W-AKKT system (\ref{eq:w-akkt}) is equivalent to
\begin{equation}\label{eq:LMweak_KKT}
\langle D_u\theta(u^*), v\rangle_\mathcal{U}  + \uplim\limits_{k\rightarrow +\infty}\langle \lambda^{k}, Sv + t(\zeta - Su^*)\rangle \leq 0\quad \forall\ v\in \mathcal{U},\ \forall\ \zeta\in K,\ \ t\geq 0.
\end{equation}
\end{lemma}
\begin{proof}
It is obvious that the W-AKKT system (\ref{eq:w-akkt}) implies (\ref{eq:LMweak_KKT}). 

Now we prove that if (\ref{eq:LMweak_KKT}) holds, so does the W-AKKT system (\ref{eq:w-akkt}).
Taking $t=0$, we have
$$ \langle D_u\theta(u^*), v\rangle_\mathcal{U} + \uplim\limits_{k\rightarrow +\infty}\langle \lambda^{k}, Sv \rangle \leq 0\quad \forall\ v\in \mathcal{U}.$$
If we change $v$ to $-v$, we have
$$ \langle D_u\theta(u^*), -v\rangle_\mathcal{U} - \lowlim\limits_{k\rightarrow +\infty}\langle \lambda^{k}, Sv \rangle = \langle D_u\theta(u^*), -v\rangle_\mathcal{U} + \uplim\limits_{k\rightarrow +\infty}\langle \lambda^{k}, S(-v) \rangle \leq 0.$$
These two inequalities give
$$ \uplim\limits_{k\rightarrow +\infty}\langle \lambda^{k}, Sv\rangle -  \lowlim\limits_{k\rightarrow +\infty}\langle \lambda^{k}, Sv \rangle \leq 0.$$
On the other hand, by the definitions of $\uplim\limits_{k\rightarrow +\infty}$ and $\lowlim\limits_{k\rightarrow +\infty}$, we have
$$ \uplim\limits_{k\rightarrow +\infty}\langle \lambda^{k}, Sv\rangle \geq  \lowlim\limits_{k\rightarrow +\infty}\langle \lambda^{k}, Sv \rangle.$$
Therefore,
$$ \uplim\limits_{k\rightarrow +\infty}\langle \lambda^{k}, Sv\rangle = \lowlim\limits_{k\rightarrow +\infty}\langle \lambda^{k}, Sv \rangle$$
and
$$\langle D_u\theta(u^*), v\rangle_\mathcal{U} + \lim\limits_{k\rightarrow +\infty}\langle \lambda^{k}, Sv \rangle = 0\quad \forall\ v\in \mathcal{U}.$$

By taking $v = 0$ and $t=1$, we have
$$-\uplim\limits_{k\rightarrow +\infty}\langle \lambda^{k}, \zeta - Su^*\rangle \geq 0\quad \forall\ \zeta\in K.$$
This completes the proof.
\end{proof}

Let $\zeta^* = Su^*$ and 
$$ \mathcal{C}(S, \zeta^*, K) = \{Sv + t(\zeta - \zeta^*)\in \mathcal{X}:\ \forall\ v\in \mathcal{U},\ \forall\ \zeta\in K\ \ \mbox{and}\ \ t\geq 0\}.$$
Note that $\mathcal{C}(S, \zeta^*, K)$ is a convex cone in $\mathcal{X}$.

\begin{theorem}\label{thm:R_KKT}
Assume that there exists a norm $\|\cdot\|_*$ such that $(\mathcal{C}(S, \zeta^*, K), \|\cdot\|_*)$ is a Banach space and it is continuously embedded into $(\mathcal{X}, \|\cdot\|)$. Then there exists $\tilde{\lambda}\in (\mathcal{C}(S, \zeta^*, K), \|\cdot\|_*)'$ such that
\begin{equation}\label{eq:R_plm}
\langle D_u\theta(u^*), v\rangle_\mathcal{U}  + \langle \tilde{\lambda}, Sv + t(\zeta - \zeta^*)\rangle_* \leq 0\quad \forall\ v\in \mathcal{U},\ \forall\ \zeta\in K\ \ \mbox{and}\ \ t\geq 0,
\end{equation}
 or equivalently,
\begin{equation}\label{eq:R_KKT}
\left\{
\begin{aligned}
&\langle D_u\theta(u^*), v\rangle_\mathcal{U}  + \langle \tilde{\lambda}, Sv\rangle_* = 0\quad \forall\ v\in \mathcal{U};\\
& -\langle \tilde{\lambda}, \zeta - Su^*\rangle_* \geq 0\quad \forall\ \zeta\in K,
\end{aligned}
\right.
\end{equation}
where $\langle\cdot, \cdot\rangle_*$ is the duality pair of $(\mathcal{C}(S, \zeta^*, K), \|\cdot\|_*)'$ and $(\mathcal{C}(S, \zeta^*, K), \|\cdot\|_*)$. 
\end{theorem}

\begin{proof}
  Let $\{\lambda^k\}_{k=1}^{+\infty}$ be the same as that in the W-AKKT system (\ref{eq:w-akkt}). Since $(\mathcal{C}(S, \zeta^*, K), \|\cdot\|_*)$ is a Banach space and it is continuously embedded into $(\mathcal{X}, \|\cdot\|)$, $\{\lambda^k\}\subset (\mathcal{C}(S, \zeta^*, K), \|\cdot\|_*)'$.

Since $\mathcal{C}(S, \zeta^*, K)$ is a linear space, for any $v\in \mathcal{U}$, $\zeta \in K$ and $t\geq 0$, we have 
$$-[Sv + t(\zeta - \zeta^*)]\in  \mathcal{C}(S, \zeta^*, K),$$
i.e., there exist $w\in \mathcal{U}$, $\bar{\zeta}\in K$ and $s\geq 0$ such that
$$-[Sv + t(\zeta - \zeta^*)] = Sw + s(\zeta - \zeta^*).$$
Therefore, by (\ref{eq:LMweak_KKT}), we have
$$ \langle D_u\theta(u^*), v\rangle_\mathcal{U}  + \uplim\limits_{k\rightarrow +\infty}\langle \lambda^{k}, Sv + t(\zeta - \zeta^*)\rangle \leq 0 $$
and
\begin{equation}\nonumber
\begin{aligned}
&\quad\  \langle D_u\theta(u^*), w\rangle_\mathcal{U}  - \uplim\limits_{k\rightarrow +\infty}\langle \lambda^{k}, Sv + t(\zeta - \zeta^*)\rangle\\
&\leq \langle D_u\theta(u^*), w\rangle_\mathcal{U}  - \lowlim\limits_{k\rightarrow +\infty}\langle \lambda^{k}, Sv + t(\zeta - \zeta^*)\rangle\\
& = \langle D_u\theta(u^*), w\rangle_\mathcal{U}  + \uplim\limits_{k\rightarrow +\infty}\langle \lambda^{k}, Sw + s(\bar{\zeta} - \zeta^*)\rangle\\
& \leq 0,
\end{aligned}
\end{equation}
which follows
$$ \langle D_u\theta(u^*), w\rangle_\mathcal{U} \leq \uplim\limits_{k\rightarrow +\infty}\langle \lambda^{k}, Sv + t(\zeta - \zeta^*)\rangle \leq -\langle D_u\theta(u^*), v\rangle_\mathcal{U}.$$
Furthermore, we have 
$$ |\uplim\limits_{k\rightarrow +\infty}\langle \lambda^{k}, Sv + t(\zeta - \zeta^*)\rangle| \leq \|D_u\theta(u^*)\|_{\mathcal{U}}(\|v\|_{\mathcal{U}} + \|w\|_{\mathcal{U}}) < + \infty\quad \forall\ v\in \mathcal{U},\ \forall \zeta \in K.$$
It follows that there exists a subsequence of $\{\lambda^k\}_{k=1}^{+\infty}$ (which will still be denoted by  $\{\lambda^k\}_{k=1}^{+\infty}$), such that
$$|\langle \lambda^{k}, \eta \rangle_*| = |\langle \lambda^{k}, \eta \rangle|< +\infty\quad \forall\ \eta \in \mathcal{C}(S, \zeta^*, K),\ \forall\ k = 1, 2, \dots.$$
According to the uniform boundedness principle,  $\{\lambda^k\}_{k=1}^{+\infty}$ is bounded in\\
 $(\mathcal{C}(S, \zeta^*, K), \|\cdot\|_*)'$. Then there exists a subsequence of $\{\lambda^k\}_{k=1}^{+\infty}$ (which will still be denoted by  $\{\lambda^k\}_{k=1}^{+\infty}$) converges to $\tilde{\lambda} \in (\mathcal{C}(S, \zeta^*, K), \|\cdot\|_*)'$ in the weak-* topology of $(\mathcal{C}(S, \zeta^*, K), \|\cdot\|_*)'$. Therefore,
$$ \lim\limits_{k\rightarrow +\infty}\langle \lambda^k, \eta\rangle = \langle \tilde{\lambda}, \eta\rangle_*\quad \forall\ \eta \in \mathcal{C}(S, \zeta^*, K).$$

By (\ref{eq:LMweak_KKT}), we have (\ref{eq:R_plm}) holds. By taking $t = 0$ and $v = 0$, we can get the equivalence of (\ref{eq:R_plm}) and (\ref{eq:R_KKT}). This completes the proof.
\end{proof}

\begin{Remark}
Note that $\mathcal{X} = \mathcal{C}(S, \zeta^*, K)$, which means Robinson’s condition (\ref{eq:RCQ}) holds, is a special case of Theorem \ref{thm:R_KKT}.
\end{Remark}

\begin{Remark}
If there exists $\lambda_0$ such that 
$$\langle \lambda_0, Sv + t(\zeta - \zeta^*)\rangle \leq \uplim\limits_{k\rightarrow +\infty}\langle \lambda^{k}, Sv + t(\zeta - \zeta^*)\rangle  \quad \forall\ v\in \mathcal{U},\ \forall\ \zeta\in K\ \ \mbox{and}\ \ t\geq 0,$$
then we have
$$\langle D_u\theta(u^*), v\rangle_\mathcal{U}  + \langle \lambda_0, Sv + t(\zeta - Su^*)\rangle \leq 0\quad \forall\ v\in \mathcal{U},\ \forall\ \zeta\in K\ \ \mbox{and}\ \ t\geq 0. $$
This implies $\lambda_0$ is a proper Lagrange multiplier. 

Therefore, every weak accumulation point (if exists) of $\{\lambda^k\}_{k=1}^{+\infty}$ is a proper Lagrange multiplier, and furthermore, the weak convergence of the multipliers generated by the classical ALM in $\mathcal{X}$ implies the existence of a proper Lagrange multiplier. However, the existence of proper Lagrange multipliers can not guarantee the convergence of the multipliers generated by the classical ALM in $\mathcal{X}$, see e.g., Example \ref{Exm:ALM} in this paper. 

This will also be useful for developing new constrained qualification conditions to guarantee the existence of the proper Lagrange multiplier.
\end{Remark}

\section{Conclusions}
\label{sec:con}
In this paper, we systematically developed a new decomposition framework to investigate Lagrange multipliers of the KKT system  of constrained optimization problems and variational inequalities in Hilbert spaces. Our new framework is totally different from existing frameworks based on separation theorems. We derived the weak form asymptotic KKT system and introduced the essential Lagrange multiplier. The basic theory of the essential Lagrange multiplier has been established in this paper. The existence theory of the essential Lagrange multiplier shows essential differences in finite and infinite-dimensional cases. Based on it, we also gave necessary and sufficient conditions for the existence and uniqueness of the proper Lagrange multiplier. The results theoretically confirm the necessity of using the asymptotic or approximate KKT system in the infinite-dimensional case as well. We proved the convergence of the classical augmented Lagrangian method without using the information of Lagrange multipliers of the problem, and essentially characterized the convergence properties of the  multipliers generated by the classical augmented Lagrangian method. Some sufficient conditions to guarantee the existence of the essential Lagrange multiplier and the proper Lagrange multiplier were also derived. We also answered some fundamental mathematical questions related to constrained optimization, which further improve the mathematical foundation of constrained optimization. 

\appendix
\section{The proof of Theorem \ref{thm:w-akkt}}\label{append:A}
We first prove that if a feasible point $u^*$ satisfies the W-AKKT system (\ref{eq:w-akkt}), then $u^*$ is the global minimizer of the model problem (\ref{eq:sm_g}). According to the convex optimization theory (cf. \cite[Proposition 26.5]{bauschke2011convex}, \cite[Proposition 2.1, Chapter II]{ekeland1999convex} or \cite[Corollary 4.20]{jahn2020introduction}), it suffices to show that 
$$ \langle D_u\theta(u^*), v - u^*\rangle_\mathcal{U} \geq 0\quad\ \forall\ Sv\in K.$$

Since $u^*$ satisfies the W-AKKT system (\ref{eq:w-akkt}), there exists $\{\lambda^k\}_{k=1}^{+\infty}\subset \mathcal{X}$ such that 
$$\langle D_u\theta(u^*), v - u^*\rangle_\mathcal{U} + \lim\limits_{k\rightarrow+\infty}\langle S^*\lambda^k, v - u^*\rangle = 0\quad \forall\ v\in\mathcal{U} $$
and
$$ -\uplim\limits_{k\rightarrow+\infty}\langle \lambda^k, \zeta - Su^*\rangle \geq 0\quad \forall\ \zeta\in K.$$
 
Therefore,
$$ -\uplim\limits_{k\rightarrow+\infty}\langle \lambda^k, Sv - Su^*\rangle \geq 0\quad \forall\ Sv \in K,$$
i.e.,
$$ -\uplim\limits_{k\rightarrow+\infty}\langle S^*\lambda^k, v - u^*\rangle_{\mathcal{U}} \geq 0\quad \forall\ Sv \in K.$$
It follows
\begin{equation}\label{eq:thm_if}
\begin{aligned}
\langle D_u\theta(u^*), v - u^*\rangle_\mathcal{U} 
&= - \lim\limits_{k\rightarrow+\infty}\langle S^*\lambda^k, v - u^*\rangle_{\mathcal{U}}\\
& \geq -\uplim\limits_{k\rightarrow+\infty}\langle S^*\lambda^k, v - u^*\rangle_{\mathcal{U}}  \geq 0\quad \forall\ Sv \in K,
\end{aligned}
\end{equation}
which implies that $u^*$ is the global minimizer of the model problem (\ref{eq:sm_g}).

For the other direction, we will use the classical ALM for the model problem as an optimization procedure regularization to prove it. The algorithm will be given and the convergence analysis will be carried out without using any assumptions on Lagrange multipliers of the model problem. The convergence analysis borrows some ideas of that in  \cite{glowinski2020admm} and \cite{jiao2016alternating}  for ADMM. 

\begin{Remark}
A comprehensive discussion of the classical ALM in Banach spaces based on a different approach can be found in \cite[Chapter 4]{steck2018lagrange}, and a further application of these results to develop new constraint qualification conditions in Banach spaces can be found in \cite{borgens2020new}.
\end{Remark}

We first rewrite the model problem (\ref{eq:sm_g}) into the following equivalent form
\begin{equation}\label{eq:GOP}
\min\limits_{u\in \mathcal{U}, \zeta\in \mathcal{X}}\theta(u) + I_K(\zeta)\ \ \mbox{s.t.}\ \ Su = \zeta,
\end{equation}
where $I_K(\cdot)$ is the indicator function of $K$, which is defined by
\begin{equation}\nonumber
I_K(\zeta) = \left\{
\begin{aligned}
&0,&\quad &\mbox{if}\ \zeta\in K,\\
&+\infty,&\quad &\mbox{if}\ \zeta\not\in K.
\end{aligned}
\right.
\end{equation}

Note that the global minimizer of the problem (\ref{eq:GOP}) is given by $(u^*, \zeta^*)$, where $u^*$ is the global minimizer of the model problem (\ref{eq:sm_g}) and $\zeta^* = Su^*$.

\subsection{The classical augmented Lagrangian method}
The augmented Lagrangian functional $L_{\beta}(u,\zeta;\lambda): (\mathcal{U}\times \mathcal{X})\times \mathcal{X}\rightarrow \mathbb{R}\cup \{+\infty\}$ of the model problem (\ref{eq:sm_g}) based on (\ref{eq:GOP}) is given by
\begin{equation}\nonumber
L_{\beta}(u,\zeta;\lambda) = \theta(u) + I_K(\zeta) + \langle\lambda, Su - \zeta \rangle + \frac{\beta}{2}\|Su - \zeta\|^2,
\end{equation}
where $\beta>0$.
The classical augmented Lagrangian method for the model problem (\ref{eq:sm_g}) based on this augmented Lagrangian functional is given by
\begin{equation}\label{eq:ALM}
\left\{
\begin{aligned}
(u^{k+1}, \zeta^{k+1}) & = \arg\min\limits_{(u,\zeta)\in\mathcal{U}\times \mathcal{X}}L_{\beta}(u, \zeta;\lambda^k),\\
\lambda^{k+1} & = \lambda^k + \beta(Su^{k+1} - \zeta^{k+1}),
\end{aligned}
\right.
\end{equation}
where $\lambda^1\in \mathcal{X}$ is given.

\subsection{Convergence analysis}In this part we will give an elementary proof of the convergence of the algorithm. The proof is composed of several steps. 

We first give a characterization of the iterators $\{(u^{k}, \zeta^k, \lambda^k)\}_{k=1}^{+\infty}$ by the first order optimality system of the subproblem in the first step. 
Since the subproblem in the first step is a convex problem, solving the subproblem is equivalent to solve
\begin{equation}\nonumber
\left\{
\begin{aligned}
& D_u\theta(u^{k+1}) + S^*[\lambda^k + \beta(Su^{k+1} - \zeta^{k+1})] = 0;\\
& I_K(\zeta) - I_K(\zeta^{k+1}) - \langle\lambda^k + \beta(Su^{k+1} - \zeta^{k+1}) ,\zeta - \zeta^{k+1}\rangle\geq 0\quad \forall\ \zeta\in \mathcal{X}.
\end{aligned}
\right.
\end{equation}
Hence, the iterators of the classical ALM satisfy
\begin{equation}\label{eq:FOD_ALM}
\left\{
\begin{aligned}
& D_u\theta(u^{k+1}) + S^*\lambda^{k+1} = 0;\\
& I_K(\zeta) - I_K(\zeta^{k+1}) - \langle\lambda^{k+1} ,\zeta - \zeta^{k+1}\rangle\geq 0\quad \forall\ \zeta\in \mathcal{X};\\
& \beta r^{k+1} = \lambda^{k+1} - \lambda^k,
\end{aligned}
\right.
\end{equation}
where $r^{k+1} = Su^{k+1} - \zeta^{k+1}$. According to (\ref{eq:FOD_ALM}), it holds
\begin{equation}\label{eq:lambda_zeta}
\lambda^k\in \partial I_K(\zeta^k)\quad \forall\ k=2, 3, \dots.
\end{equation}
Without loss of generality, we also assume that (\ref{eq:lambda_zeta}) holds for $k=1$.

\subsubsection{Convergence of $\{(u^k, \zeta^k)\}_{k=1}^{+\infty}$} We first prove that $\{\|r^{k}\|\}_{k=1}^{+\infty}$ is  Fej\'{e}r monotone. 
It follows from (\ref{eq:FOD_ALM}) that
\begin{equation}\nonumber
\beta S^*r^{k+1} 
= S^*( \lambda^{k+1} - \lambda^k)
= S^* \lambda^{k+1} - S^* \lambda^k
= D_u\theta(u^{k}) - D_u\theta(u^{k+1})
\end{equation}
and
\begin{equation}\nonumber
\begin{aligned}
\beta\langle r^{k+1}, \zeta^{k+1} - \zeta^k\rangle
 & = \langle \lambda^{k+1} - \lambda^k, \zeta^{k+1} - \zeta^k \rangle\\
 & = I_K(\zeta^k) - I_K(\zeta^{k+1}) - \langle\lambda^{k+1} ,\zeta^k - \zeta^{k+1}\rangle\quad (\geq 0)\\
 &\quad + I_K(\zeta^{k+1}) - I_K(\zeta^{k}) - \langle\lambda^{k} ,\zeta^{k+1} - \zeta^{k}\rangle\quad (\geq 0)\\
 &\geq 0,
\end{aligned}
\end{equation}
which yield
\begin{equation}\nonumber
\begin{aligned}
\beta\langle r^{k+1} - r^{k}, r^{k+1}\rangle
& = \beta\langle S(u^{k+1} - u^{k}), r^{k+1} \rangle - \beta\langle \zeta^{k+1} - \zeta^{k}, r^{k+1} \rangle\\
& = \langle u^{k+1} - u^{k}, \beta S^*r^{k+1} \rangle_\mathcal{U} - \beta\langle r^{k+1}, \zeta^{k+1} - \zeta^{k} \rangle\\
& \leq -\langle u^{k+1} - u^{k}, D_u\theta(u^{k+1}) - D_u\theta(u^{k}) \rangle_\mathcal{U} \\
& \leq - c_0\|u^{k+1} - u^{k}\|^2_{\mathcal{U}},\\
\end{aligned}
\end{equation}
where we used (\ref{eq:theta_sc}) in the last inequality. Furthermore, by applying the identity
$$ \beta\langle r^{k+1} - r^{k}, r^{k+1}\rangle = \frac{\beta}{2}\left(\|r^{k+1}\|^2 -\|r^{k}\|^2 + \|r^{k+1} - r^k\|^2\right),$$
we have
$$ \frac{\beta}{2}\left(\|r^{k+1}\|^2 -\|r^{k}\|^2 + \|r^{k+1} - r^k\|^2\right) \leq - c_0\|u^{k+1} - u^{k}\|^2_{\mathcal{U}}$$
or equivalently
\begin{equation}\label{eq:Fejer_m}
\|r^{k+1}\|^2 \leq  \|r^{k}\|^2 - \|r^{k+1} - r^k\|^2 - 2c_0\beta^{-1}\|u^{k+1} - u^{k}\|^2_{\mathcal{U}}\leq \|r^{k}\|^2,
\end{equation}
which shows that $\{\|r^{k}\|\}_{k=1}^{+\infty}$ is Fej\'{e}r monotone.

Secondly, we prove that
$$\sum\limits_{k= 1}^{+\infty}\|r^{k}\|^2< +\infty.$$

We denote by $\eta = (u,\zeta)\in \mathcal{U}\times \mathcal{X}$ and recall the Bregman distance induced by the convex functional $\theta(u) + I_K(\zeta)$ at $(D_u\theta(u), \lambda)$ with $\lambda\in \partial I_K(\zeta)$, which is given by
\begin{equation}\label{eq:Bregman}
\mathcal{D}_{\eta}(\hat{\eta},\eta; \lambda): = \theta(\hat{u}) - \theta(u) - \langle D_u\theta(u), \hat{u} - u\rangle_\mathcal{U} + I_K(\hat{\zeta}) - I_K(\zeta) - \langle \lambda, \hat{\zeta} - \zeta\rangle\geq 0,
\end{equation}
for any $\hat{\eta}\in \mathcal{U}\times \mathcal{X}$.

By the assumption (\ref{eq:theta_sc}) of $\theta(u)$, we have
\begin{equation}\label{eq:Breg_Elliptic}
\mathcal{D}_{\eta}(\hat{\eta},\eta; \lambda) 
 \geq \theta(\hat{u}) - \theta(u) - \langle D_u\theta(u), \hat{u} - u\rangle_\mathcal{U}
 \geq \frac{c_0}{2}\|\hat{u} - u\|_{\mathcal{U}}^2.
\end{equation}

Let $\hat{\eta} = (\hat{u}, \hat{\zeta})$ satisfy $\hat{\zeta} = S\hat{u} \in K$. We will always assume this in the following analysis. It follows from (\ref{eq:FOD_ALM}), (\ref{eq:lambda_zeta}) and (\ref{eq:Bregman}) that 
\begin{equation}\label{eq:Breg_lu}
\begin{aligned}
&\quad\ \mathcal{D}_{\eta^{k}}(\hat{\eta},\eta^{k}; \lambda^{k}) - \mathcal{D}_{\eta^{n}}(\hat{\eta},\eta^{n}; \lambda^{n}) + \mathcal{D}_{\eta^{n}}(\eta^{k},\eta^{n}; \lambda^{n})\\
& = \langle \lambda^n - \lambda^{k}, Su^{k} - S\hat{u}\rangle+ \langle  \lambda^n - \lambda^{k}, S\hat{u} - \zeta^{k}\rangle\\
& = \langle \lambda^n - \lambda^{k}, Su^{k} - \zeta^{k}\rangle 
= -\sum\limits_{i = n}^{k-1}\langle \lambda^{i+1} - \lambda^{i}, r^{k}\rangle
= -\beta\sum\limits_{i = n}^{k-1}\langle r^{i+1}, r^k\rangle,\\
\end{aligned}
\end{equation}
for any $k>n$. 
This gives
$$ \mathcal{D}_{\eta^{k}}(\eta^{k+1},\eta^{k}; \lambda^{k}) + \beta\|r^{k+1}\|^2  = \mathcal{D}_{\eta^{k}}(\hat{\eta},\eta^{k}; \lambda^{k}) - \mathcal{D}_{\eta^{k+1}}(\hat{\eta},\eta^{k+1}; \lambda^{k+1})\quad \forall\ k = 1, 2, \dots.$$
Summing over $k$ from $n$ to $m\ (> n)$ on both sides, we have
\begin{equation}\label{eq:Breg_pre}
\sum\limits_{k=n}^{m}\left( \mathcal{D}_{\eta^{k}}(\eta^{k+1},\eta^{k}; \lambda^{k}) + \beta\|r^{k+1}\|^2\right) = \mathcal{D}_{\eta^{n}}(\hat{\eta},\eta^{n}; \lambda^{n}) - \mathcal{D}_{\eta^{m}}(\hat{\eta},\eta^{m}; \lambda^{m}).
\end{equation}
Then by taking $n=1$, using (\ref{eq:Breg_Elliptic}) and letting $m\rightarrow +\infty$, we have
\begin{equation}\nonumber
\sum\limits_{k=1}^{+\infty}\left( \mathcal{D}_{\eta^{k}}(\eta^{k+1},\eta^{k}; \lambda^{k}) + \beta\|r^{k+1}\|^2\right) \leq \mathcal{D}_{\eta^{1}}(\hat{\eta},\eta^{1}; \lambda^{1})<+\infty,
\end{equation}
which yields
\begin{equation}\label{eq:r_c}
\sum\limits_{k=1}^{+\infty}\|r^{k}\|^2<\infty\quad \mbox{and}\quad \{\mathcal{D}_{\eta^{k}}(\hat{\eta},\eta^{k}; \lambda^{k})\}_{k=1}^{+\infty}\ \ \mbox{is\ a\ Cauchy\ sequence},
\end{equation}
by (\ref{eq:Breg_pre}).
Moreover, we have
\begin{equation}\label{eq:S_c}
\|Su^k - \zeta^k\| = \|r^k\|\ \rightarrow 0\ \ \mbox{as}\ \ k\rightarrow + \infty.
\end{equation}

Now, we prove that $\{u^k\}_{k=1}^{+\infty}$ and $\{\zeta^k\}_{k=1}^{+\infty}$ are Cauchy sequences.
For any $k>n$, it follows from  (\ref{eq:Fejer_m}) that
\begin{equation}\label{eq:rm_b}
\begin{aligned}
|\beta\sum\limits_{i = n}^{k-1}\langle r^{i+1}, r^k\rangle|
\leq \frac{\beta}{2}\sum\limits_{i = n}^{k-1}(\|r^{i+1}\|^2 +  \|r^k\|^2)
\leq \beta\sum\limits_{i = n}^{k}\|r^{i}\|^2.
\end{aligned}
\end{equation}

For any $k>n$, by (\ref{eq:Breg_Elliptic}), (\ref{eq:Breg_lu}), (\ref{eq:rm_b}) and (\ref{eq:r_c}),  we have
\begin{equation}\nonumber
\begin{aligned}
\frac{c_0}{2}\|u^k - u^n\|_{\mathcal{U}}^2
& \leq \mathcal{D}_{\eta^{n}}(\eta^k,\eta^{n}; \lambda^{n})\\
& = -\beta\sum\limits_{i = n}^{k-1}\langle r^{i+1}, r^k\rangle - [\mathcal{D}_{\eta^{k}}(\hat{\eta},\eta^{k}; \lambda^{k}) - \mathcal{D}_{\eta^{n}}(\hat{\eta},\eta^{n}; \lambda^{n})]\\
& \leq  \beta\sum\limits_{i = n}^{k}\|r^{i}\|^2 + |\mathcal{D}_{\eta^{k}}(\hat{\eta},\eta^{k}; \lambda^{k}) - \mathcal{D}_{\eta^{n}}(\hat{\eta},\eta^{n}; \lambda^{n})|\ \rightarrow\ 0\ \mbox{as}\ n, k\rightarrow +\infty.\\
\end{aligned}
\end{equation}

Therefore, $\{u^k\}_{k=1}^{+\infty}$ is a Cauchy sequence and $\{\zeta^k\}_{k=1}^{+\infty}$ is a Cauchy sequence by (\ref{eq:S_c}). 

Let $\bar{u}\in \mathcal{U}$ and $\bar{\zeta}\in\mathcal{X}$ be the limits of $\{u^k\}_{k=1}^{+\infty}$ and $\{\zeta^k\}_{k=1}^{+\infty}$ respectively, i.e., 
\begin{equation}\label{eq:uzeta_c}
\lim\limits_{k\rightarrow + \infty}u^k = \bar{u}\quad \mbox{and}\quad \lim\limits_{k\rightarrow + \infty}\zeta^k = \bar{\zeta}.
\end{equation}

Since $\{ \zeta^k\}_{k=1}^{+\infty}\subset K$ and $K$ is closed, we have $\bar{\zeta}\in K$.  It follows from (\ref{eq:S_c}) that
\begin{equation}\label{eq:C_eq}
S\bar{u} = \bar{\zeta}.
\end{equation}

\subsubsection{Global convergence of the algorithm} Now we prove that $(\bar{u}, \bar{\zeta})$ is the global minimizer of the problem (\ref{eq:GOP}). Since 
$$ \theta(u) + I_K(\zeta)$$
is lower semicontinuous, we have
\begin{equation}\label{eq:O_l}
\theta(\bar{u}) + I_K(\bar{\zeta})\leq \lowlim\limits_{k\rightarrow +\infty}[\theta(u^k) + I_K(\zeta^k)].
\end{equation}

Note that  $\lambda^k\in \partial I_K(\zeta^k)$. The convexity of the objective functional gives
$$ \theta(\bar{u}) + I_K(\bar{\zeta}) - [\theta(u^k) + I_K(\zeta^k)] - [ \langle \lambda^k, \bar{\zeta} - \zeta^k\rangle + \langle D_u\theta(u^k), \bar{u} - u^k \rangle_\mathcal{U}]\geq 0, $$
i.e.,
\begin{equation}\label{eq:upk}
\theta(u^k) + I_K(\zeta^k) \leq \theta(\bar{u}) + I_K(\bar{\zeta}) - [ \langle \lambda^k, \bar{\zeta} - \zeta^k\rangle + \langle D_u\theta(u^k), \bar{u} - u^k \rangle_\mathcal{U}].
\end{equation}
For any fixed $n$  and $k>n$, it follows from (\ref{eq:C_eq}) and (\ref{eq:Breg_lu}) with $\hat{\eta} = (\bar{u}, \bar{\zeta})$ that
\begin{equation}\label{eq:crucial_eq}
\begin{aligned}
&\quad\ |\langle \lambda^k, \bar{\zeta} - \zeta^k\rangle + \langle D_u\theta(u^k), \bar{u} - u^k \rangle_\mathcal{U}|\\
& =  |\langle \lambda^k - \lambda^n, \bar{\zeta} - \zeta^k\rangle + \langle D_u\theta(u^k) - D_u\theta(u^n), \bar{u} - u^k \rangle_\mathcal{U} \\
&\quad + \langle \lambda^n, \bar{\zeta} - \zeta^k\rangle + \langle D_u\theta(u^n), \bar{u} - u^k \rangle_\mathcal{U}|\\
& = | \beta\sum\limits_{i = n}^{k-1}\langle r^{i+1}, r^k\rangle + [\langle \lambda^n, \bar{\zeta} - \zeta^k\rangle + \langle D_u\theta(u^n), \bar{u} - u^k \rangle_\mathcal{U}]|\\
& \leq \beta\sum\limits_{i = n}^{k}\|r^{i}\|^2 + |\langle \lambda^n, \bar{\zeta} - \zeta^k\rangle + \langle D_u\theta(u^n), \bar{u} - u^k \rangle_\mathcal{U}|.
\end{aligned}
\end{equation}
Hence, by (\ref{eq:upk}), (\ref{eq:crucial_eq}), (\ref{eq:r_c}) and (\ref{eq:uzeta_c}), we get
\begin{equation}\nonumber
\begin{aligned}
&\quad \uplim\limits_{k\rightarrow +\infty}[\theta(u^k) + I_K(\zeta^k)]\\ 
& \leq \theta(\bar{u}) + I_K(\bar{\zeta}) + \uplim\limits_{k\rightarrow +\infty}\left\{\beta\sum\limits_{i = n}^{k}\|r^{i}\|^2 + |\langle \lambda^n, \bar{\zeta} - \zeta^k\rangle + \langle D_u\theta(u^n), \bar{u} - u^k \rangle_\mathcal{U}| \right\}\\
& =  \theta(\bar{u}) + I_K(\bar{\zeta}) + \beta\sum\limits_{i = n}^{+\infty}\|r^{i}\|^2.
\end{aligned}
\end{equation}
If we further take $n\rightarrow + \infty$ and use (\ref{eq:r_c}), we obtain
\begin{equation}\label{eq:O_u}
\lim\limits_{k\rightarrow +\infty}\sup[\theta(u^k) + I_K(\zeta^k)] \leq \theta(\bar{u}) + I_K(\bar{\zeta}).
\end{equation}
This together with (\ref{eq:O_l}) yields
\begin{equation}\label{eq:O_c}
\theta(\bar{u}) + I_K(\bar{\zeta}) = \lim\limits_{k\rightarrow +\infty}[\theta(u^k) + I_K(\zeta^k)].
\end{equation}

For any $(\hat{u},\hat{\zeta})$ satisfing $S\hat{u} = \hat{\zeta}$, we will prove that
$$ \theta(\bar{u}) + I_K(\bar{\zeta})\leq  \theta(\hat{u}) + I_K(\hat{\zeta}).$$

The proof is similar to that of (\ref{eq:O_u}). Since
$$ \theta(u^k) + I_K(\zeta^k) \leq \theta(\hat{u}) + I_K(\hat{\zeta}) - [ \langle \lambda^k, \hat{\zeta} - \zeta^k\rangle + \langle D_u\theta(u^k), \hat{u} - u^k \rangle_\mathcal{U}]$$
and
\begin{equation}\nonumber
\begin{aligned}
 |\langle \lambda^k, \hat{\zeta} - \zeta^k\rangle + \langle D_u\theta(u^k), \hat{u} - u^k \rangle_\mathcal{U}|
 \leq \beta\sum\limits_{i = n}^{k}\|r^{i}\|^2 + |\langle \lambda^n, r^k\rangle|
\end{aligned}
\end{equation}
where $n$ is fixed and $k>n$, we have
\begin{equation}\nonumber
\begin{aligned}
\theta(u^k) + I_K(\zeta^k) 
& \leq \theta(\hat{u}) + I_K(\hat{\zeta}) - [ \langle \lambda^k, \hat{\zeta} - \zeta^k\rangle + \langle D_u\theta(u^k), \hat{u} - u^k \rangle_\mathcal{U}] \\
& \leq \theta(\hat{u}) + I_K(\hat{\zeta}) + \beta\sum\limits_{i = n}^{k}\|r^{i}\|^2 + |\langle \lambda^n, r^k\rangle|.
\end{aligned}
\end{equation}
By (\ref{eq:r_c}) and (\ref{eq:O_c}), we get
\begin{equation}\nonumber
\begin{aligned}
 \theta(\bar{u}) + I_K(\bar{\zeta}) = \lim\limits_{k\rightarrow+\infty}\theta(u^k) + I_K(\zeta^k) \leq \theta(\hat{u}) + I_K(\hat{\zeta}) + \beta\sum\limits_{i = n}^{+\infty}\|r^{i}\|^2.
\end{aligned}
\end{equation}
By taking $n\rightarrow + \infty$ and using (\ref{eq:r_c}) again, it leads to
\begin{equation}\nonumber
\theta(\bar{u}) + I_K(\bar{\zeta}) \leq \theta(\hat{u}) + I_K(\hat{\zeta}),
\end{equation}
which implies $(\bar{u},\bar{\zeta})$ is the global minimizer of the model problem.

\subsection{Deriving the weak form asymptotic KKT system} Now we derive the weak form asymptotic KKT system (\ref{eq:w-akkt}). This can be achieved by exploring the limiting case of (\ref{eq:FOD_ALM}).

We have proved that $\{u^k\}_{k=1}^{+\infty}$ strongly converges to the global minimizer $u^*$ of the model problem (\ref{eq:sm_g}). Since $\theta(u)$ is continuously differentiable, it gives
$$ \lim\limits_{k\rightarrow +\infty} D_u\theta(u^k) = D_u\theta(u^*)$$
and
\begin{equation}\label{eq:w-akkt1}
\langle D_u\theta(u^*), v\rangle_\mathcal{U}  + \lim\limits_{k\rightarrow +\infty}\langle \lambda^{k}, Sv\rangle = 0\quad \forall\ v\in \mathcal{U},
\end{equation}
by (\ref{eq:FOD_ALM}).

In order to get (\ref{eq:w-akkt}), it suffices to show that
$$-\uplim\limits_{k\rightarrow + \infty}\langle \lambda^k, \zeta - Su^*\rangle \geq 0\quad \forall\ \zeta\in K.$$

Denote by $\zeta^* = Su^* \in K$.  According to (\ref{eq:FOD_ALM}), for any $\zeta\in K$, we have
\begin{equation}\nonumber
-\langle\lambda^{k+1} ,\zeta - \zeta^{k+1} \rangle \geq 0
\end{equation}
and
\begin{equation}\label{eq:w-pre1}
\begin{aligned}
- \langle\lambda^{k+1} ,\zeta - \zeta^*\rangle
& \geq \langle \lambda^{k+1}, \zeta^* - \zeta^{k+1}\rangle +  \langle D_u\theta(u^{k+1}), u^* - u^{k+1} \rangle_\mathcal{U} \\
&\quad\quad \quad\quad\quad \quad\quad\quad \ \ \  -  \langle D_u\theta(u^{k+1}), u^* - u^{k+1} \rangle_\mathcal{U}.
\end{aligned}
\end{equation}

For any fixed $n$  and $k + 1>n$, applying (\ref{eq:crucial_eq}) to $(u^*, \zeta^*)$, we have 
\begin{equation}\nonumber
\begin{aligned}
&\quad |\langle \lambda^{k+1}, \zeta^* - \zeta^{k+1}\rangle + \langle D_u\theta(u^{k+1}), u^* - u^{k+1} \rangle_\mathcal{U}| \\
& \leq \beta\sum\limits_{i = n}^{k+1}\|r^{i}\|^2 + |\langle \lambda^n, \zeta^* - \zeta^{k+1}\rangle + \langle D_u\theta(u^n), u^* - u^{k+1} \rangle_\mathcal{U}|.
\end{aligned}
\end{equation}
Following the same arguments as before (see (\ref{eq:crucial_eq})-(\ref{eq:O_u})), we get
\begin{equation}\label{eq:w-pre2}
|\langle \lambda^{k+1}, \zeta^* - \zeta^{k+1}\rangle + \langle D_u\theta(u^{k+1}), u^* - u^{k+1} \rangle_\mathcal{U}| \rightarrow 0\quad \mbox{as}\quad k\rightarrow +\infty.
\end{equation}

On the other hand, the strong convergence of $\{u^{k}\}_{k=1}^{+\infty}$ and the boundedness of $D_u\theta(u)$ around $u^*$ give
\begin{equation}\label{eq:w-pre3}
\lim\limits_{k\rightarrow +\infty}\langle D_u\theta(u^{k+1}), u^* - u^{k+1} \rangle_\mathcal{U} = 0.
\end{equation}

Together with (\ref{eq:w-pre1}), (\ref{eq:w-pre2}) and (\ref{eq:w-pre3}), we arrive at 
\begin{equation}\label{eq:w-akkt2}
-\uplim\limits_{k\rightarrow + \infty}\langle \lambda^k, \zeta - \zeta^*\rangle = \lowlim\limits_{k\rightarrow + \infty}[-\langle \lambda^k, \zeta - \zeta^*\rangle]  \geq 0\quad \forall\ \zeta\in K.
\end{equation}

Therefore, by (\ref{eq:w-akkt1}), (\ref{eq:w-akkt2}) and $\zeta^* = Su^*$, we have the weak form asymptotic KKT system (\ref{eq:w-akkt}) at $u^*$.

This completes the proof of Theorem \ref{thm:w-akkt}.


\bibliographystyle{amsplain}
\bibliography{KKT_ref}

 \medskip

\end{document}